\documentclass[11pt,reqno]{article}
\usepackage{authblk}
\usepackage{hyperref}
\usepackage{amssymb}
\usepackage{amsthm}
\usepackage{amsmath}
\usepackage{multirow}
\usepackage{float}
\usepackage{graphicx}
\usepackage{listings}
\usepackage{MnSymbol}
\usepackage{nicematrix}
\usepackage{adjustbox}
\usepackage{booktabs}
\newcommand{\subsubsubsection}[1]{%
  \paragraph{#1}\mbox{}\\
}

\newtheorem{theorem}{Theorem}[section]

\newtheorem{proposition}[theorem]{Proposition}

\newtheorem{remark}[theorem]{Remark}

\newtheorem{example}{Example}[section]

\theoremstyle{remark}

\begin{document}
\title{\Large\bfseries Equivalence Problem for Non-Linearizable Fourth-Order ODEs with Five-Dimensional Lie Symmetry subalgebra via Inductive Cartan Equivalence Method}
\author[1]{Sondos R. Khalil}
\author[1]{Ahmad Y. Al-Dweik\thanks{aaldweik@birzeit.edu}}
\author[1]{Marwan Aloqeili\thanks{maloqeili@birzeit.edu}}
\author[2]{F. M. Mahomed\thanks{Corresponding Author: Fazal.Mahomed@wits.ac.za}}
\affil[1]{Department of Mathematics, Birzeit University, Ramallah, Palestine}
\affil[2]{School of Computer Science and Applied Mathematics, 
University of the Witwatersrand, Johannesburg, Wits 2050,  South Africa}
\maketitle
\begin{abstract}
 Four coframes of invariant 1-forms are explicitly constructed using the Inductive Cartan equivalence method with rank zero corresponding to four distinct branches. These coframes are employed to characterize non-linearizable fourth-order ODEs under point transformation with a five-point symmetry Lie subalgebra. Moreover, we propose a procedure for obtaining the point transformation by using the derived invariant coframes, demonstrated through examples.
\end{abstract}
\section{Introduction}
The Lie symmetry classification \cite{Lie second ode} over the complex domain identifies four distinct classes encompassing all scalar ordinary differential equations (ODEs) of order two,  $u^{\prime\prime} = f(x,u,u^{\prime})$, that admit a  Lie algebra of dimension three.

By adopting the \textit{direct method}, Ibragimov and Meleshko \cite{Ibragimov second ode, Meleshko second ode} provided an invariant classification of these equations by determining the canonical forms for each class and analyzed them extensively. They showed that any non-linearizable second-order ODE having a three-dimensional Lie symmetry algebra can be mapped, through point transformations, into one of these canonical forms.

In more recent studies, Al-Dweik et al. \cite{Sondos} applied the \textit{Cartan equivalence method} to construct explicitly three-dimensional invariant coframes corresponding to the three branches. They characterized these ODEs geometrically and a procedure was developed for deriving the corresponding point transformations.

Abuloha et al. \cite{abuloha2026} naturally extended the results in \cite{Sondos} to the equivalence problem of  third-order ODEs which are non-linearizable by point transformation and possess four-dimensional Lie point symmetry subalgebras.  A direct extension of the recent research is to study the equivalence problem for non-linearizable fourth-order ODEs that admit five-dimensional Lie symmetry subalgebras, considering transformations that are point and contact. These form two separate equivalence problems, as some ODEs that are non-linearizable by point transformation may become linearizable under contact transformation.

 Many studies in the literature provide a full characterization of scalar fourth-order ODEs that possess real Lie algebras $L_n $ of dimension $n$, where $n$ ranges from 1 to 8. Mahomed in \cite{F Mohamad1989} derived the realizations of three-dimensional Lie algebras in the real plane. Cerquetelli \cite{Cerquetelli2002} and Fatima \cite{ Fatima2013} provided a full list of scalar fourth‐order ODEs that have 4‐dimensional Lie algebras. In \cite{Shah2021,Shah2024}, Shah et. al derived a complete classification of scalar 4th-order ODEs that admit 5-, 6- and 8-dimensional Lie point symmetry algebras. In \cite{Shah2021}, Shah et. al, showed that there are no fourth-order differential equations possessing a seven-dimensional Lie algebra as its symmetry algebra.
 
 Classification of scalar fourth-order ODEs possessing real Lie point symmetry algebras $(m,n)$, where  $m$ is the class of the algebra and $n= 5, 6, 8$  is its dimension, is given in the Appendix I (see \cite{Shah2021}).
As indicated in Table $5$ of Appendix I, we infer that any non-linearizable fourth-order ODE  $u^{(4)} = f(x, u, u^{\prime}, u^{\prime \prime},u^{\prime \prime \prime})$ via point transformation having a five-dimensional
Lie point symmetry subalgebra falls under twelve different types listed in Table $1$. Throughout this paper, we denote $ u^{\prime}, u^{\prime \prime},u^{\prime \prime \prime}$ by $p, ~q, ~r$.
\begin{table}[H]
\begin{adjustbox}{width=\columnwidth,center}
\begin{tabular}{|c|c|c|}
\hline Algebra Type & Generators & The corresponding fourth-order equations \\
\hline $(24,5), \alpha=0$ & $\partial_x, \partial_u, x \partial_x+\alpha u \partial_u, x \partial_u, x^2 \partial_u$ & $u^{(4)}={ r}^ {\frac{4}{3}}$ \\
\hline $(24,5), \alpha=b+1$ & $\partial_x, \partial_u, x \partial_x+\alpha u \partial_u, x \partial_u, x^2 \partial_u$ & $u^{(4)}= r^{\frac{b-3}{b-2}},~b\neq 1,2,3$\\
\hline $(24,5), \alpha=2$ & $\partial_x, \partial_u, x \partial_x+\alpha u \partial_u, x \partial_u, x^2 \partial_u$ & $u^{(4)}= r^{2}$ \\
\hline $(24,5), \alpha=1$ & $\partial_x, \partial_u, x \partial_x+\alpha u \partial_u, x \partial_u, x^2 \partial_u$ & $u^{(4)}= r^{ \frac{3}{2}}$ \\
\hline $(24,5), \alpha=\frac{b+1}{b}$ & $\partial_x, \partial_u, x \partial_x+\alpha u \partial_u, x \partial_u, x^2 \partial_u$ & $u^{(4)}= r^{\frac{1-3 b}{1-2 b}},b\neq 0,\frac{1}{3},\frac{1}{2},1 $ \\
\hline $(5,5)$ & $\partial_x, \partial_u, x \partial_x-u \partial_u, u \partial_x, x \partial_u$ & $u^{(4)}=\frac{5}{3} \frac{r^2}{q}+ q^{ \frac{5}{3}}$ \\
\hline $(27,5),r=1$ & $\partial_x, \partial_u, x \partial_u,2x \partial_x+u\partial u, x^2 \partial_x+xu\partial u$ & $u^{(4)}=\frac{4}{3} \frac{r^2}{q}+ q^{ \frac{7}{3}}$ \\
\hline $(15,5)$ & $\partial_x, \partial_u, x \partial_x, u \partial_u, x^2 \partial_x$ & $
u^{(4)}=\frac{6 q r}{p}-\frac{6 q^ 3}{p^ 2}+K \frac{\left(3q^2-2pr\right)^{\frac{3}{2}}}{p^ 2}
$\\
\hline $(25,5), r=3$ & $\partial_x, \partial_u, x \partial_u, x^2 \partial_u, x \partial_x+\left(3 u+x^3\right) \partial_u$ & $
u^{(4)}= e^{r}
$ \\
\hline $(26,5), r=1$ & $\partial_x, \partial_u, x \partial_u, x \partial_x, u \partial_u$ & $
u^{(4)}=K \frac{r^ 2}{q},~K \neq 0,\frac{4}{3},\frac{5}{3}$ \\
\hline $(6,6)$ & $\partial_x, \partial_u, x \partial_x, u \partial_x, x \partial_u, u \partial_u$ & $
u^{(4)}=\frac{5}{3} \frac{r^ 2}{q}
$ \\
\hline $(28,6),r=1$ & $\partial_x, \partial_u, x \partial_x,x \partial_u, u \partial_u, x^2 \partial_x+xu\partial u$ & $
u^{(4)}=\frac{4}{3} \frac{r^ 2}{q}
$ \\
\hline
\end{tabular}
\end{adjustbox}
\caption{Canonical forms of non-linearizable fourth-order ODEs under point transformation admitting
five-dimensional Lie point symmetry subalgebras. }
\end{table}
 
 \begin{remark} In Table 1, we omit all linear canonical forms listed in Table
$4$ of the Appendix I, as our focus is on non-linearizable fourth-order ODEs.
\end{remark}

\begin{remark}
 \begin{itemize}
 In Table~1, we leave out $K$ in the following  representative equations
 \item $(24,5)$, $\alpha = 0$, 
listed in Table~5 of Appendix I, where it is noted that the transformation 
$\bar{x} = x,\; \bar{u} = \frac{1}{K^3}u$ 
maps the canonical form 
$\bar{u}^{(4)} = K\bar{u}'''^{\frac{4}{3}}$ 
to 
$u^{(4)} = u'''^{\frac{4}{3}}$.

\item $(24,5)$, $\alpha = b+1$, 
listed in Table $5$ of Appendix I, where it is observed that the transformation 
$\bar{x} = \frac{1}{K}x,\; \bar{u} = \frac{1}{K^3}u$ 
maps the canonical form 
$\bar{u}^{(4)} = K\bar{u}'''^{\frac{b-3}{b-2}}$ 
to 
$u^{(4)} = u'''^{\frac{b-3}{b-2}}$.
\item  $(24,5)$, $\alpha = 2$, 
listed in Table $5$ of Appendix I, where it is noted that the transformation 
$\bar{x} = x,\; \bar{u} = \frac{1}{K}u$ 
maps the canonical form 
$\bar{u}^{(4)}= K\bar{u}'''^{2}$ 
to 
$u^{(4)} = u'''^{2}$.
\item  $(24,5)$, $\alpha = 1$, 
listed in Table $5$ of Appendix I, where it is seen that the transformation 
$\bar{x} =x,\; \bar{u} = \frac{1}{K^2}u$ 
maps the canonical form 
$\bar{u}^{(4)} = K\bar{u}'''^{\frac{3}{2}}$ 
to 
$u^{(4)} = u'''^{\frac{3}{2}}$.

\item  $(24,5)$, $\alpha = \frac{b+1}{b}$, 
listed in Table $5$ of Appendix I, where it is noted that the transformation 
$\bar{x} = K^2x,\; \bar{u} = K^{\frac{3}{b}}u$ 
maps the canonical form 
$\bar{u}^{(4)} = K\bar{u}'''^{\frac{1-3b}{1-2b}}$ 
to 
$u^{(4)} = u'''^{\frac{1-3b}{1-2b}}$.
\item  $(24,5)$, $r = 3$, 
listed in Table $5$ of Appendix I, where it is observed that the transformation 
$\bar{x} = \frac{-6}{K}x,\; \bar{u} = \frac{1296}{K^3}u$ 
maps the canonical form 
$\bar{u}^{(4)}= Ke^{-\bar{u}^{\prime \prime \prime}/6}$ 
to 
$u^{(4)} = e^{u'''}$.
\item $(5,5)$,  
listed in Table $5$ of Appendix I, where it is noticed that the transformation 
$\bar{x} =x,\; \bar{u} = K^{-\frac{3}{2}}u$ 
maps the canonical form 
$\bar{u}^{(4)} =\frac{5}{3} \frac{\bar{u}^{\prime \prime \prime2}}{\bar{u}^{\prime \prime}}+K\bar{ u}^{\prime \prime \frac{5}{3}} $ 
to 
$u^{(4)} = \frac{5}{3} \frac{u^{\prime \prime \prime2}}{u^{\prime \prime}}+ u^{\prime \prime \frac{5}{3}}$.
\item $(27,5),r=1$,  
listed in Table $5$ of Appendix I, where it is seen that the transformation 
$\bar{x} = x,\; \bar{u} = K^{-\frac{3}{4}}u$ 
maps the canonical form 
$\bar{u}^{(4)} =\frac{4}{3} \frac{\bar{u}^{\prime \prime \prime2}}{\bar{u}^{\prime \prime}}+K\bar{ u}^{\prime \prime \frac{7}{3}} $ 
to 
$u^{(4)} = \frac{4}{3} \frac{u^{\prime \prime \prime2}}{u^{\prime \prime}}+ u^{\prime \prime \frac{7}{3}}$.

\end{itemize}
\end{remark}

 The objective of this work is to present an invariant characterization of non-linearizable fourth-order ODEs admitting a five-dimensional Lie point symmetry subalgebra, as listed in Table 1.
  
A principal challenge in this work lies in the implementation of Cartan’s Equivalence Method, primarily due to the well-known issue of expression swell. In this context, the associated coframes become highly intricate, rendering the computational process extremely demanding. 
To address this challenge, we  employ the \textit{Inductive Cartan Equivalence Method}. This method significantly simplifies the resulting expressions. Despite this simplification, the computations still remain complicated. Consequently, we further introduce a novel framework for Cartan’s method based on a branching strategy via differential relative invariants. This is complemented by introducing a chain of auxiliary functions to systematically manage and simplify the resulting expressions.

This article is arranged as follows. In the following section, the Inductive Cartan’s equivalence method is applied to scalar fourth-order ODEs via the action of the Lie group of point transformations, employed to obtain four invariant coframes associated with four distinct branches. These coframes are then utilised to characterize all fourth-order ODEs presented in Table $1$. Section $3$ outlines the main theorems. Section $4$ outlines a procedure for constructing corresponding point transformations from the constructed invariant coframes, illustrative examples are presented to demonstrate the principal theorems and the paper concludes with a summary.

\section{Implementation of Cartan’s equivalence method}
\setcounter{equation}{0}
We refer the interested reader to the scholarly works \cite{Olver1995,Neut2003} for necessary definitions, notation and standard ideas employed in this section. Throughout this paper, the 1-forms $\pi_i,~i=1,...,12$ denote the modified Mourer-Cartan forms derived from the general solution of the linear absorption system.

Let $\mathbf{J}^3$ denote the third-order jet space equipped with local coordinates $\left(x, u, p=u', q=u'', r=u'''\right) \in \mathbb{R}^5$. The following defines an adapted coframe on $M=\mathbf{J}^3$.
\begin{equation}\label{2.1}
\left(\begin{array}{l}
\omega^1\\\omega^2\\\omega^3\\\omega^4\\\omega^5\\
\end{array}\right)
=\Omega
\left(\begin{array}{c}
du-pdx\\
dp-qdx\\
dq-rdx\\
dr-fdx\\
dx\\
\end{array}\right),
\end{equation}
where  $\Omega$ denotes the nonsingular transition matrix given by
\begin{equation}\label{2.2}
\Omega=\left(
\begin{array}{ccccc}
1&0&0&0&0\\
-\frac{1}{2}I_{0}&1&0&0&0\\
\frac{3}{10}(I_1+I_2)&-\frac{1}{2}I_0&1&0&0\\
0&\frac{1}{10}(-3I_1+7I_2)&0&1&0\\
0&0&0&0&1\\
\end{array}
\right),
\end{equation}
wherein $I_0,~I_1, I_2$ are functions of $x,u,p,q,r$ and their explicit values in terms of $f$ are
\begin{equation}\label{2.3}
\begin{aligned}
I_0=-f_r,\quad I_1=-\frac{1}{6}\hat{D}_xI_0+\frac{1}{4}I_0^2,\quad I_2=-\frac{1}{2}\hat{D}_xI_0-f_q, 
\end{aligned}
\end{equation}
 where $\hat{D}_x=\frac{\partial}{\partial x}+p\frac{\partial}{\partial u}+q\frac{\partial}{\partial p}+r\frac{\partial }{\partial q}+f\frac{\partial}{\partial r}$. The invertible matrix $\Omega$ can be obtained from the solution of the equivalence problem under fiber-preserving transformation as explained briefly in Appendix II.
 
In local coordinates, the equivalence  of
\begin{equation}\label{2.4}
\begin{array}{l}
u^{(4)}=f\left(x, u, u^{\prime},u^{\prime \prime},u^{\prime \prime\prime}\right), \quad \bar{u}^{(4)}=\bar{f}\left(\bar{x}, \bar{u}, \bar{u}^{\prime},\bar{u}^{\prime \prime},\bar{u}^{\prime \prime \prime}\right),
\end{array}
\end{equation}
by means of the point transformation
\begin{equation}\label{2.5}
\bar{x}=\varphi \left( x,u\right),~\bar{u} =\psi \left( x,u \right),\quad \phi_x\psi_u-\phi_u\psi_x \neq 0\\
\end{equation}relative to (\ref{2.1}), the following equivalence relations arise
\begin{equation}\label{2.6}
\Phi^*\left(\begin{array}{c}
\bar{\omega}^1 \\
\bar{\omega}^2 \\
\bar{\omega}^3\\
\bar{\omega}^4\\
\bar{\omega}^5
\end{array}\right)=\left(\begin{array}{ccccc}
a_1 & 0 & 0 &0 &0\\
a_2 & a_3 & 0 &0&0 \\
a_4 & a_5 & a_6& 0&0\\
a_7 & a_8&a_9&a_{10}&0\\
a_{11}&&0&0&a_{13}
\end{array}\right)\left(\begin{array}{c}
\omega^1 \\
\omega^2 \\
\omega^3\\
\omega^4\\
\omega^5
\end{array}\right),
\end{equation}
for functions $a_i(x,u,p,q,r),~ i=1,2 \dots, 11,13,$  where $\Phi^*$  is the pullback operator corresponding to the map arising from the third prolongation of the point transformation \eqref{2.5}. It follows that the resulting structure group is the twelve-dimensional Lie group 
\begin{equation}\label{2.7}
G=\left\{
\left(\begin{array}{ccccc}
a_1 & 0 & 0 &0&0\\
a_2 & a_3 & 0&0&0 \\
a_4 & a_5 & a_6&0&0\\
a_7&a_8&a_9&a_{10}&0\\
a_{11}&&0&0&a_{13}
\end{array}\right)  \Bigg \vert a_1a_3a_6a_{10}a_{13}\neq 0
\right\}.
\end{equation}
We demonstrate that by applying Cartan’s equivalence method to the equivalence problem between scalar fourth-order ODEs and each canonical form in Table 1 yields four invariant coframes of rank zero.

Now, let us define \begin{equation}\label{coframe}(\theta^1,~\theta^2,~\theta^3,~\theta^4,~\theta^5)^{T}=S(\omega^1,\omega^2,\omega^3,\omega^4,\omega^5)^{T}\end{equation} to be the lifted coframe, where $S$ belongs to $G$.

Once absorption is performed, the first structure equation reads
\begin{equation}\label{e2.9}
\begin{array}{ll}
d\left(\begin{array}{c}
\theta^1 \\
\theta^2 \\
\theta^3\\
\theta^4\\
\theta^5
\end{array}\right)=\left(\begin{array}{ccccc}
\pi^{\prime}_1 & 0 & 0 &0&0\\
\pi^{\prime}_2&  \pi^{\prime}_3 & 0&0&0 \\
\pi^{\prime}_4&  \pi^{\prime}_5&\pi^{\prime}_6&0&0\\
\pi^{\prime}_{7}&\pi^{\prime}_8&\pi^{\prime}_9&\pi^{\prime}_{10}&0\\
\pi^{\prime}_{11}&0&0&0&\pi^{\prime}_{12}
\end{array}\right) \wedge\left(\begin{array}{c}
\theta^1 \\
\theta^2 \\
\theta^3\\
\theta^4\\
\theta^5
\end{array}\right)+\left(\begin{array}{c}
T_{25}^1 ~\theta^2 \wedge \theta^5 \\
T_{35}^2 ~\theta^3\wedge\theta^5\\
T_{45}^3~ \theta^4 \wedge \theta^5 \\
0\\
0\\
\end{array}\right),
\end{array}
\end{equation}

The exact formulas for the essential torsion coefficients are\begin{equation}\label{2.10} T_{25}^1=-\frac{a_1}{a_3a_{13}},~T^{2}_{35}=-\frac{a_3}{a_6a_{13}},~ T^{3}_{45}=-\frac{a_6}{a_{10}a_{13}},\end{equation} that can be set to $-1$ via normalization of the group parameters\begin{equation}\label{2.11}\begin{aligned}a_3=\frac{a_1} {a_{13}}, ~a_6=\frac{a_1}{a_{13}^2},~ a_{10}=\frac{a_1}{a_{13}^3}.\end{aligned}\end{equation} Consequently, the structure group reduces to
\begin{equation}\label{2.12}
G_1=\left\{
\left(\begin{array}{ccccc}
a_1 & 0 & 0&0&0  \\
a_2 & \frac{a_1}{a_{13}} & 0 & 0&0 \\
a_4 & a_5 & \frac{a_1}{a_{13}^2} &0&0\\
a_7&a_8&a_9&\frac{a_1}{a_{13}^3}&0\\
a_{11}&0&0&0&a_{13}
\end{array}\right)
 \Bigg \vert a_1a_{13}\neq 0
\right\},
\end{equation}
with the lifted one-forms
\begin{equation}\label{2.13}
\left(\begin{array}{l}
\theta^1 \\
\theta^2 \\
\theta^3 \\
\theta^4\\
\theta^5
\end{array}\right)=\left(\begin{array}{ccccc}
a_1 & 0 & 0 & 0& 0 \\
a_2 & \frac{a_1}{a_{13}} & 0 &0&0 \\
a_4 & a_5 & \frac{a_1}{a_{13}^2} &0&0\\
a_7&a_8&a_9&\frac{a_1}{a_{13}^3}&0\\
a_{11}&0&0&0&a_{13}
\end{array}\right)\left(\begin{array}{c}
\omega^1 \\
\omega^2 \\
\omega^3\\
\omega^4\\
\omega^5
\end{array}\right) .
\end{equation}
Proceeding to the \textit{second iteration} of the reduction process, the absorbed structure equations read as follows
\begin{equation}\label{2.14}
d\left(\begin{array}{c}
\theta^1 \\
\theta^2 \\
\theta^3\\
\theta^4\\
\theta^5
\end{array}\right)=\left(\begin{array}{ccccc}
 \pi^{\prime}_1 & 0 & 0 &0&0\\
\pi^{\prime}_2 &  \pi^{\prime}_1-\pi^{\prime}_{9} & 0&0&0 \\
\pi^{\prime}_3 & \pi^{\prime}_4& \pi^{\prime}_{1}-2\pi^{\prime}_{9}&0&0\\
\pi^{\prime}_{ 5}&\pi^{\prime}_6&\pi^{\prime}_7&\pi^{\prime}_1-3\pi^{\prime}_{9}&0\\
\pi^{\prime}_8&0&0&0&\pi^{\prime}_{9}
\end{array}\right) \wedge\left(\begin{array}{c}
\theta^1 \\
\theta^2 \\
\theta^3\\
\theta^4\\
\theta^5
\end{array}\right)+\left(\begin{array}{c}
- \theta^2 \wedge \theta^5 \\
-\theta^3\wedge\theta^5 \\
T^{3}_{35}~\theta^3\wedge\theta^5-\theta^4\wedge\theta^5\\
T^{4}_{45}~\theta^4\wedge\theta^5\\
0
\end{array}\right),
\end{equation}
The exact formula for the essential torsion coefficients are

\begin{equation} T_{35}^3=\frac{a_9a^2_{13}-3a_5a_{13}+3a_2}{a_1},~  T^{4}_{45}=\frac{-a_9a_{13}^2-3a_5a_{13}+5a_2}{a_1}.
\end{equation}
This can can be set to $0$ via normalization of the group parameters
  \begin{equation*}\begin{aligned}
a_2=a_9a_{13}^2,\quad  a_5=\frac{4}{3}a_9a_{13}.\end{aligned}\end{equation*}
Therefore, the structure group is now given by $G_2$ by incorporating the values of $a_2$ and $a_5$ in $G_1$, that gives the adapted coframe \eqref{coframe} with $S \in G_2$. 

In the \textit{third iteration} of the reduction scheme, absorption leads the structure equations to become
\begin{equation}\label{2.17}
\begin{aligned}
d\!\left(\!\begin{array}{c}
\theta^1 \\[2pt]
\theta^2 \\[2pt]
\theta^3 \\[2pt]
\theta^4 \\[2pt]
\theta^5
\end{array}\!\right)
&=\setlength{\arraycolsep}{1pt}
\left(\!\begin{array}{ccccc}
\pi'_1 & 0 & 0 & 0 & 0\\
\pi'_2 & \pi'_1 - \pi'_7 & 0 & 0 & 0\\
\pi'_3 & \tfrac{4}{3}\pi'_2 & \pi'_1 - 2\pi'_7 & 0 & 0\\
\pi'_4 & \pi'_5 & \pi'_2 & \pi'_1 - 3\pi'_7 & 0\\
\pi'_6 & 0 & 0 & 0 & \pi'_7
\end{array}\!\right)
\wedge
\left(\!\begin{array}{c}
\theta^1 \\[2pt]
\theta^2 \\[2pt]
\theta^3 \\[2pt]
\theta^4 \\[2pt]
\theta^5
\end{array}\!\right)
+
\left(\!\begin{array}{c}
-\theta^2 \wedge \theta^5 \\[3pt]
-\theta^3 \wedge \theta^5 \\[3pt]
\sum\limits^5_{i=4}T^3_{2i}\,\theta^2 \wedge \theta^i - \theta^4 \wedge \theta^5 \\[3pt]
\sum\limits^5_{i=4}T^4_{3i}\,\theta^3 \wedge \theta^i \\[3pt]
T^5_{25}~\theta^2\wedge\theta^5
\end{array}\!\right),
\end{aligned}
\end{equation}

where 
\begin{equation}\label{2.18}
\begin{aligned}
T_{24}^3 &= \frac{1}{3}T_{34}^4 = \frac{a_{13}^2}{a_1} I_4,
&\qquad
T_{25}^3 &= \frac{a_8 a_{13}}{a_1}-\frac{7}{3}\frac{a_4}{a_1}
          +\frac{8}{9}\frac{a_9^2 a_{13}^4}{a_1^2}, \\[6pt]
T_{35}^4 &= -\frac{a_8 a_{13}}{a_1}-\frac{a_4}{a_1}
          +\frac{4}{3}\frac{a_9^2 a_{13}^4}{a_1^2},
&\qquad
T_{25}^5 &= -\frac{5}{6}\frac{a_{11}}{a_1}
          +\frac{a_{13}}{a_1} I_3
          -\frac{1}{2}\frac{a_9 a_{13}^4}{a_1^2} I_4 ,
\end{aligned}
\end{equation}

and \begin{equation}
I_ 3=-\frac{1}{8} {I _0}_r I_ 0+\frac{1}{4} {I_ 1}_r-\frac{1}{12} {I_ 2}_r,\quad I _4=-\frac{1}{6} {I_ 0}_r.
\end{equation}

When a factor multiplying an absolute invariant does not rely on the group parameters, it gives rise to a relative invariant. Under a point transformation, such a quantity is mapped to a constant scalar of itself. In particular, expressions that vanish prior to the transformation continue to vanish thereafter.

Clearly $I_4$ is a relative invariant so we have two branches.

\subsection{Branch $I_4 =0$.}
The canonical form $(15,5),~K=0$ is the only canonical form in Table $1 $ that belongs to this branch, so we can translate $T^3_{25},~T^{4}_{35},~T^5_{25}$ to $0$ by normalizing 
\begin{equation}\begin{aligned}
a_4=\frac{2}{3}\frac{a_9^2a_{13}^4}{a_1},\quad a_8=\frac{2}{3}\frac{a_9^2a_{13}^3}{a_1},\quad a_{11}=\frac{6}{5}a_{13}I_3.
\end{aligned}\end{equation}
Hence, the structure group is now given by  $G_3$ by incorporating the values of $a_4,~a_8, ~a_{11}$ in $G_2$, that yields the adapted coframe \eqref{coframe} with $S \in G_3$.

In the \textit{fourth iteration} of the reduction scheme, absorption leads to the structure equations becoming
 \begin{equation}\label{I44}
\begin{aligned}
d\!\left(\!\begin{array}{c}
\theta^1 \\[2pt]
\theta^2 \\[2pt]
\theta^3 \\[2pt]
\theta^4 \\[2pt]
\theta^5
\end{array}\!\right)
&=\setlength{\arraycolsep}{1pt}
\left(\!\begin{array}{ccccc}
\pi'_1 & 0 & 0 & 0 & 0\\
\pi'_2 & \pi'_1-\pi'_4 & 0 & 0 & 0\\
0 & \tfrac{4}{3}\pi'_2 & \pi'_1-2\pi'_4 & 0 & 0\\
\pi'_3 & 0 & \pi'_2 & \pi'_1-3\pi'_4 & 0\\
0& 0 & 0 & 0 & \pi'_4
\end{array}\!\right)
\wedge
\left(\!\begin{array}{c}
\theta^1 \\[2pt]
\theta^2 \\[2pt]
\theta^3 \\[2pt]
\theta^4 \\[2pt]
\theta^5
\end{array}\!\right)
+\setlength{\arraycolsep}{1pt}
\left(\!\begin{array}{c}
- \theta^2 \wedge \theta^5 \\[3pt]
-\theta^3 \wedge \theta^5 \\[3pt]
T^{3}_{15}~\theta^1 \wedge \theta^5-\theta^4 \wedge \theta^5\\[3pt]
 T^4_{25}\theta^2\wedge \theta^5\\[3pt]
 \sum\limits^3_{i=2}T^5_{1i}\,\theta^1 \wedge \theta^i+T^5_{15}\, \theta^1 \wedge \theta^5 
\end{array}\!\right).
\end{aligned}
\end{equation}

The explicit values for $T^3_{15},~T^4_{25},~ T^5_{12},~T^5_{13},~T^5_{15}$ can be given as 

 \begin{equation}
 \begin{aligned}
 T^3_{15}&=\frac{I_7}{a_{13}^3}+\frac{a_7}{a_1}-\frac{2}{9}\frac{a_9^3a_{13}^6}{a_{1}^3},~~\quad T^4_{25}=-\frac{I_7}{a_{13}^3}-\frac{a_7}{a_1}+\frac{2}{9}\frac{a_9^3a_{13}^6}{a_{1}^3}+\frac{I_8}{a_{13}^3}\\
 T^5_{12}&=\frac{a_{13}^2I_6}{a_{1}^2}-\frac{4}{3}\frac{a_9a_{13}^5I_9}{a_1^3},\qquad T^5_{13}=\frac{a_{13}^3I_9}{a_1^2},\qquad T^5_{15}=\frac{I_5}{a_1},
 \end{aligned}
 \end{equation}
where 
\begin{equation}\begin{aligned} 
I_5&=-\frac{18}{5}I_0I_3-\frac{3}{5}{I_0}_p+\frac{2}{5}(I_2-3I_1)_q,\quad I_6=-\frac{36}{25}I_3^2-\frac{3}{5}I_0{I_3}_q-\frac{6}{5}{I_3}_p,\\
I_7&=\frac{1}{2}I_0^3-\frac{3}{10}(\hat{D}_x(I_1+I_2)-f(I_1+I_2)_r)+\frac{6}{5}fI_3+\frac{1}{20}I_0(I_2-39I_1),\\
I_8&=\frac{7}{8}I_0^3-(\hat{D}_xI_2-{I_2}_rf)-f_p+6fI_3+I_0(-3I_1-\frac{1}{2}I_2),\quad I_9=-\frac{6}{5}{I_3}_q.
\end{aligned}
\end{equation}
 
It is noted that $T^3_{15}+T^4_{25}=\frac{I_8}{a_{13}^3}$, which means that $I_8$ is a relative invariant. For the canonical form $u^{(4)}=\frac{6pr}{p}-\frac{6q^3}{p^2}$ we need to go through the sub-branch $I_5=0,~I_8=0,~I_9=0, ~I_6 \neq 0 $, so that we can translate $T^4_{25}$ to $0$ and $T^5_{12}$ to $1$ by normalizing \begin{equation}a_1=a_{13}J_6,\quad a_7=\frac{2}{9}\frac{a_9^3a_{13}^4}{J_6^2}-\frac{J_6I_7}{a_{13}^2},\end{equation}
 where $J^2_6=I_6.$
 
 In the \textit{fifth iteration} of the reduction scheme, absorption leads the structure equations to become, \begin{equation}\label{I4}
\begin{aligned}
d\!\left(\!\begin{array}{c}
\theta^1 \\[2pt]
\theta^2 \\[2pt]
\theta^3 \\[2pt]
\theta^4 \\[2pt]
\theta^5
\end{array}\!\right)
&=\setlength{\arraycolsep}{0.5pt}
\left(\!\begin{array}{ccccc}
\pi'_1 & 0 & 0 & 0 & 0\\
\pi'_2 & 0 & 0 & 0 & 0\\
0 & \tfrac{4}{3}\pi'_2 & -\pi'_1 & 0 & 0\\
0 & 0 & \pi'_2 & -2\pi'_1 & 0\\
0& 0 & 0 & 0 & \pi'_1
\end{array}\!\right)
\wedge
\left(\!\begin{array}{c}
\theta^1 \\[2pt]
\theta^2 \\[2pt]
\theta^3 \\[2pt]
\theta^4 \\[2pt]
\theta^5
\end{array}\!\right)
+\setlength{\arraycolsep}{0.5pt}
\left(\!\begin{array}{c}
- \theta^2 \wedge \theta^5 \\[3pt]
T^2_{25}\,\theta^2 \wedge \theta^5-\theta^3 \wedge \theta^5 \\[3pt]
T^3_{23}\,\theta^2\wedge \theta^3+T^3_{35}\,\theta^3\wedge \theta^5-\theta^4 \wedge \theta^5\\[3pt]
 \sum\limits^5_{i=2}T^4_{1i}\,\theta^1 \wedge \theta^i+\sum\limits^4_{j=3}T^4_{2j}\,\theta^2 \wedge \theta^j+ T^4_{45}\,\theta^4 \wedge \theta^5\\[3pt]
 \sum\limits^2_{i=1}T^5_{i5}\,\theta^i \wedge \theta^5+ \theta^1 \wedge \theta^2 
\end{array}\!\right).
\end{aligned}
\end{equation}
 The explicit values of $T^4_{15},~ T^4_{45}$ are given by
 \begin{equation}
\begin{aligned}
T^4_{45} &= \frac{I_{11}+a_9a_{13}^2}{a_{13}J_6}, 
&\qquad T^4_{15} &= \frac{I_{10}}{a_{13}^4},\\
\end{aligned}
\end{equation}
where  
\begin{equation}
 \begin{aligned}
 I_{10}&=\hat{D}_xI_7-f_u+I_0I_7+(I_1+I_2)(\frac{9}{100}(9I_1-I_2)-\frac{9}{40}I_0^2),\\
 I_{11}&=-3 {J_6}_x-3p{J_6}_u-3q{J_6}_p.
 \end{aligned}
 \end{equation}
For the canonical form $u^{(4)}= \frac{6pr}{p}-\frac{6q^3}{p^2},$ we require to go through the sub-branch $I_{10} \neq 0 $, so we can translate $T^4_{45}$ to $0$ and $T^4_{15}$ to $1$ by normalizing \begin{equation}a_9=-\frac{I_{11}}{J_{10}^2},\quad a_{13}=J_{10},\end{equation}
 where $J^4_{10}=I_{10}.$
 As a result, on space $M$, the invariant coframe is given below 
 \begin{equation}\label{InvariantI4}
 \begin{pmatrix}
 \theta^{1}\\
 \theta^{2}\\
 \theta^{3}\\ 
 \theta^4\\ 
 \theta^5\\
 \end{pmatrix}= \left(\begin{array}{ccccc} 
J_6J_{10}& 0 & 0&0&0 \\ -I_{11} &J_6 & 0&0&0 \\ \frac{2}{3}\frac{{I_{11}^2}}{J_6J_{10}}&- \frac{4}{3}\frac{I_{11}}{J_{10}} & \frac{J_6}{J_{10}}&0&0\\ \frac{-2I_{11}^3-9J_6^3I_7}{9J_6^2J_{10}^2}&\frac{2I_{11}^2}{3J_6J_{10}^2}&-\frac{I_{11}}{J_{10}^2}&\frac{J_6}{J_{10}^2}&0\\
 \frac{6}{5}I_3J_{10}&0&0&0&J_{10} \end{array}\right) 
 \begin{pmatrix} 
\omega^{1}\\
\omega^{2}\\
\omega^{3}\\ \omega^4\\
\omega^5\\
\end{pmatrix}. \end{equation} 

For the canonical form $
u^{(4)} = \frac{6pr}{p} - \frac{6q^3}{p^2},$ the structure equations associated with this invariant coframe on $M=\mathbf{J}^3$ exhibit constant structure functions.
 
Theorem \ref{thm0} follows from \cite[Theorem 8.15, page 268]{Olver1995} and is given in the subsequent section. 
 
\subsection{Branch $I_4 \neq 0$.}
In this branch, we can translate $T^3_{25},~T^4_{35},~T^5_{25} $ to $0$ and $T^3_{24}$ to $1$ by normalizing the group parameters \begin{equation}a_1=a_{13}^2I_4,\quad a_4=\frac{2}{3}\frac{a_9^2a_{13}^2}{I_4},\quad a_8=\frac{2}{3}\frac{a_9^2a_{13}}{I_4},\quad a_{11}=-\frac{3}{5}a_9a_{13}^2+\frac{6}{5}a_{13}I_3.\end{equation}

Thus, the structure group is now given by  $G_3$ by incorporating the values of $a_1,~a_4,~a_8, ~a_{11}$ in $G_2$, that results in the adapted coframe \eqref{coframe} with $S \in G_3$.

In the \textit{fourth iteration} of the reduction scheme, absorption leads the structure equations to become,
\begin{equation}\label{2.18}\setlength{\arraycolsep}{0.5pt}
\begin{aligned}
d\!\left(\!\begin{array}{c}
\theta^1 \\[2pt]
\theta^2 \\[2pt]
\theta^3 \\[2pt]
\theta^4 \\[2pt]
\theta^5
\end{array}\!\right)
&=\setlength{\arraycolsep}{0.5pt}
\left(\!\begin{array}{ccccc}
2\pi'_1 & 0 & 0 & 0 & 0\\
\pi'_2 & \pi'_1 & 0 & 0 & 0\\
0 & \tfrac{4}{3}\pi'_2 &0 & 0 & 0\\
\pi'_3 & 0 & \pi'_2 & \pi'_1 & 0\\
-\frac{3}{5}\pi'_2 & 0 & 0 & 0 & \pi'_1
\end{array}\!\right)
\wedge
\left(\!\begin{array}{c}
\theta^1 \\[2pt]
\theta^2 \\[2pt]
\theta^3 \\[2pt]
\theta^4 \\[2pt]
\theta^5
\end{array}\!\right)
+\setlength{\arraycolsep}{0pt}
\left(\!\begin{array}{c}
- \theta^2 \wedge \theta^5 \\[3pt]
\sum\limits^5_{i=3}T^2_{2i}\,\theta^2 \wedge \theta^i-\theta^3 \wedge \theta^5 \\[3pt]
\sum\limits^5_{i=3}T^3_{1i}\,\theta^1 \wedge \theta^i + T^{3}_{23}\,\theta^2\wedge \theta^3+\sum\limits^5_{j=4}T^3_{3j}\, \theta^3\wedge\theta^j +(\theta^2 +\theta^5)\wedge \theta^4\\[3pt]
\sum\limits^5_{i=4}T^4_{2i}\,\theta^2 \wedge \theta^i +T^4_{34}\theta^3\wedge \theta^4+T^4_{45}\theta^4\wedge \theta^5\\[3pt]
\sum\limits^5_{i=2}T^5_{1i}\,\theta^2 \wedge \theta^i+\sum\limits^4_{j=2}T^5_{j5}\,\theta^2 \wedge \theta^i 
\end{array}\!\right).
\end{aligned}
\end{equation}
The explicit values for $T^2_{24},~T^2_{25},~ T^3_{15}$ are given by

\begin{equation}\label{T}
\begin{array}{rcl@{\hskip 2em}rcl}
T^{2}_{24} =a_{13}I_5, \quad
T^2_{25}=\frac{1}{6}\frac{I_6}{a_{13}I_4}-\frac{1}{6}\frac{a_9}{I_4},\quad
T^{3}_{15} =\frac{I_7}{a_{13}^3}+\frac{a_7}{a_{13}^2I_4}-\frac{2}{9}\frac{a_9^3}{I_4^3},
\end{array}
\end{equation}
where 
\begin{equation}
\begin{aligned}\label{2.19}
I_5&=- \frac{1}{2}\frac{{I_4}_r}{I_4^2},
~I_6=-\frac{3}{2} I _0 I_ 4-{I_ 2}_r-6 I _3,\\
I_7&=\frac{1}{10}f\left(9I_0I_4+(-I_2+3I_1)_r-12I_3\right)+\frac{1}{2}I_0(I_0^2+\frac{1}{10}I_2)-\frac{3}{10}\hat{D}_x(I_1+I_2)-\frac{39}{20}I_0I_1.
\end{aligned}
\end{equation}
We can normalize $T^2_{25},~ T^3_{15} $ to zero by normalizing \begin{equation}\begin{aligned}a_7=-\frac{I_4I_7}{a_{13}}+\frac{2}{9}\frac{I_6^3}{a_{13}I_4^2},\quad a_{9}=\frac{I_6}{a_{13}}.\end{aligned}\end{equation}
Incorporating $a_7,\, a_9 $ in $G_3$ yields  $G_4$ as 
 $$ G_4=\left\{\left(\begin{array}{ccccc}
a_{13}^2I_4 & 0 & 0&0&0  \\
a_{13} I_{6}& a_{13}I_4 &0&0 & 0 \\
\frac{2}{3}\frac{I_6^2}{I_4}&\frac{4}{3}I_6&I_4&0&0\\
-\frac{I_4I_7}{a_{13}}+\frac{2}{9}\frac{I_6^3}{a_{13}I_4^2}&\frac{2}{3}\frac{I_6^2}{a_{13}I_4}&\frac{I_{6}}{a_{13}}&\frac{I_4}{a_{13}}&0\\
-\frac{3}{5}a_{13}I_6+\frac{6}{5}a_{13}I_3&0&0&0&a_{13}\\
\end{array}\right)\Bigg \vert a_{13}\neq 0
\right\}.$$
 It can also be observed that $I_{5}$ is a relative invariant. Consequently, we obtain the following two sub-branches.
 
\subsubsection{Sub-branch $I_5\neq 0$.}
In this sub-branch, we can translate  $T_{24}^{2}$ in equation \eqref{T} to $1$  by normalizing 
\begin{equation}
\begin{aligned}\label{2.20}
a_{13}&=\frac{1}{I_5}.
\end{aligned}
\end{equation}
That gives on space $M$ the following invariant coframe.
\begin{equation}\label{2.21}
 \begin{pmatrix}
 \theta^{1}\\
 \theta^{2}\\
 \theta^{3}\\ 
 \theta^4\\ 
 \theta^5\\
 \end{pmatrix}= \left(\begin{array}{ccccc} 
 \frac{I_4}{I_5^2} & 0 & 0&0&0 \\ \frac{I_{6}}{I_5} & \frac{I_4}{I_5} & 0&0&0 \\ \frac{2}{3}\frac{{I_6^2}}{I_4}& \frac{4}{3}I_6 & I_4&0&0\\ -I_4I_5I_7+\frac{2}{9}\frac{I_5I_6^3}{I_4^2}&\frac{2}{3}\frac{I_5I_6^2}{I_4}&I_5I_6&I_4I_5&0\\
 -\frac{3}{5}\frac{I_6}{I_5}+\frac{6}{5}\frac{I_3}{I_5}&0&0&0&\frac{1}{I_5} \end{array}\right) 
 \begin{pmatrix} 
\omega^{1}\\
\omega^{2}\\
\omega^{3}\\ \omega^4\\
\omega^5\\
\end{pmatrix}. \end{equation}

Upon evaluation of the following six canonical forms corresponding to this branch, the invariant coframe on $M=\mathbf{J}^3$ possesses structure equations with constant structure functions.
\begin{itemize}
\item The canonical form $ u^{(4)}=e^{r}$.
\item The canonical form $u^{(4)}=r^{\frac{3}{2}}$.
\item The canonical form $u^{(4)}=r^{\frac{4}{3}}$.
\item The canonical form $u^{(4)}=r^{\frac{b-3}{b-2}},~ b \neq 1,2,3$.
\item The canonical form $u^{(4)} = r^{\frac{1 - 3b}{1 - 2b}},~b \neq 0,\frac{1}{3}, \frac{1}{2},1$.
\item The canonical form 
$u^{(4)} = \frac{6 q r}{p} - \frac{6 q^3}{p^2} 
+ K \frac{ \left( 3q^2 - 2 p r \right)^{3/2}}{p^2}, 
\quad K \neq 0.$
\end{itemize}
Theorem \ref{thm1} follows from \cite[Theorem 8.15, page 268]{Olver1995} and is provided in the subsequent section.

\subsubsection{Sub-branch $I_5= 0.$}
In this sub-branch, $a_{13}$ is the only parameter left. In the \textit{fifth iteration} of the reduction scheme, absorption leads the structure equations to become

\begin{equation}\label{st fifth}
\begin{aligned}
d
\begin{pmatrix}
\theta^1\\[0.3ex]
\theta^2\\[0.3ex]
\theta^3\\[0.3ex]
\theta^4\\[0.3ex]
\theta^5
\end{pmatrix}
&=\setlength{\arraycolsep}{0.5pt}
\begin{pmatrix}
2\pi'_1 & 0      & 0      & 0       & 0\\
0       & \pi'_1& 0      & 0       & 0\\
0       & 0      & 0     & 0       & 0\\
0       & 0      & 0      & -\pi'_1 & 0\\
0       & 0      & 0      & 0       & \pi'_1
\end{pmatrix}
\wedge
\begin{pmatrix}
\theta^1\\
\theta^2\\
\theta^3\\
\theta^4\\
\theta^5
\end{pmatrix}
+\setlength{\arraycolsep}{0.5pt}
\left(
\begin{array}{c}
-\theta^{2}\wedge\theta^{5}
\\
\sum\limits^5_{i=3}T^2_{1i}\,\theta^1\wedge\theta^i+T^2_{23}\,\theta^2\wedge \theta^3+T^2_{25}\,\theta^2\wedge \theta^5-\theta^3\wedge \theta^5
\\
\sum\limits^5_{i=2}T^3_{1i}\,\theta^1\wedge \theta^i+\sum\limits^5_{j=3}T^3_{2j}\,\theta^2\wedge \theta^j+T^3_{35}\,\theta^3\wedge \theta^5-\theta^{4}\wedge\theta^{5}
\\
\sum\limits^5_{i=2}T^4_{1i}\,\theta^1\wedge \theta^i+\sum\limits^5_{j=3}T^4_{2j}\,\theta^2\wedge \theta^j+\sum\limits^5_{k=4}T^4_{3k}\,\theta^3\wedge \theta^k+T^4_{45}\,\theta^4\wedge \theta^5\\
\sum\limits^5_{i=2}T^5_{1i}\,\theta^1\wedge \theta^i+\sum\limits^3_{j=2}T^5_{j5}\,\theta^j\wedge \theta^5\\
\end{array}
\right).
\end{aligned}
\end{equation}
The explicit  value of the essential torsion  $T^2_{15} $ is \begin{equation}\begin{aligned}T^2_{15}=\frac{I_8}{a_{13}^2},\end{aligned}\end{equation} 
where $I_8=\frac{I_6}{I_4^2}\hat{D}_x(I_4-{I_4}_rf)-\frac{1}{I_4}\hat{D}_xI_6+\frac{3}{4}I_0^2-\frac{1}{2}\frac{I_0I_6}{I_4}+\frac{3}{10}I_2+\frac{1}{3}\frac{I_6^2}{I_4^2}-\frac{27}{10}I_1$.
 Clearly $I_8$ is a relative invariant and hence we have the following two 
 sub-branches.
  
 \subsubsubsection{Sub-branch $I_8 \neq 0$} 
In this sub-branch we can normalize $T^2_{15}=1$, by setting $a_{13}=J_8$, where $J_8^2=I_8$. This gives the invariant coframe as given below

\begin{equation}\label{2.20}
\left(
\begin{array}{c}\\
\theta^{1}\\
\theta^{2}\\
\theta^{3}\\
\theta^{4}\\
\theta^{5}\\\\
\end{array}
\right)
=
\left(\begin{array}{ccccc}
I_4J_8^2 & 0 & 0 & 0 & 0 \\[0.8em]
I_6J_8 & I_4J_8 & 0 & 0 & 0 \\[0.8em]
\frac{2}{3} \frac{I_6^2}{I_4} & \frac{4}{3}I_6 &I_4 & 0 & 0 \\[0.8em]
-\frac{I_4I_7}{J_8}+\frac{2}{9}\frac{I_6^3}{I_4^2J_8} &\frac{2}{3}\frac{I_6^2}{I_4J_8} &\frac{I_6}{J_8} & \frac{I_4}{J_8}&0 \\[0.8em]
-\frac{3}{5}I_6J_8+\frac{6}{5}I_3J_8 &0 &0 & 0 &J_8
\end{array}\right)
\left(
\begin{array}{c}\\
\omega^{1}\\
\omega^{2}\\
\omega^{3}\\
\omega^{4}\\
\omega^{5}\\\\
\end{array}
\right).
\end{equation}
Upon evaluation at the following four canonical forms corresponding to this branch, the invariant coframe on $M=\mathbf{J}^3$ possesses structure equations with constant structure functions.
\begin{itemize}
\item The canonical form $u^{(4)}=r^2.$
\item The canonical form $u^{(4)}=\frac{Kr^2}{q}, K \neq 0, \frac{4}{3},~\frac{5}{3}.$
\item The canonical form $u^{(4)}=\frac{5}{3}\frac{r^2}{q}+q^{\frac{5}{3}}.$
\item The canonical form $u^{(4)}=\frac{4}{3}\frac{r^2}{q}+q^{\frac{7}{3}}.$
\end{itemize}
Theorem \ref{thm2} follows from \cite[Theorem 8.15, page 268]{Olver1995} and is given in the subsequent section.

\subsubsubsection{Sub-branch $I_8 =0$.}
  In this sub-branch all relative invariants are independent of the group parameter $a_{13}$ once evaluated at the canonical forms $(6,6)$ and $((28,6),r=1),$ so the group parameter $a_{13}$ cannot be normalized. In addition, $\pi'_1$ is now uniquely defined and hence the problem becomes determinant. Consequently, the following $e$-structure is obtained on the six-dimensional prolonged space. $M^{(1)} = M\times G_4$ 
\begin{equation}\label{2.99}
\left(\begin{array}{l}
\theta^1 \\
\theta^2 \\
\theta^3 \\
\theta^4\\
\theta^5\\
\pi_1'
\end{array}\right)=
\left(\begin{array}{cccccc}
a_{13}^2I_4 & 0 & 0&0&0&0  \\
a_{13} I_{6}& a_{13}I_4 &0&0 & 0 &0\\
\frac{2}{3}\frac{I_6^2}{I_4}&\frac{4}{3}I_6&I_4&0&0&0\\
-\frac{I_4I_7}{a_{13}}+\frac{2}{9}\frac{I_6^3}{a_{13}I_4^2}&\frac{2}{3}\frac{I_6^2}{a_{13}I_4}&\frac{I_{6}}{a_{13}}&\frac{I_4}{a_{13}}&0&0\\
-\frac{3}{5}a_{13}I_6+\frac{6}{5}a_{13}I_3&0&0&0&a_{13}&0\\
I_9&I_{10}&\frac{-9I_4^2-{I_6}_r}{6I_4}&0&\frac{3I_0I_4-4I_6}{6I_4}&\frac{1}{a_{13}}
\end{array}\right)
\left(\begin{array}{c}
\omega^1 \\
\omega^2 \\
\omega^3\\
\omega^4\\
\omega^5\\
d a_{13}
\end{array}\right),
\end{equation}
where 
\begin{equation}\label{I9}
\begin{aligned}
I_9 &=  -\frac{3}{4I_4} \bigg(-\frac{2}{3}{I_0}_p I_4 - \frac{4}{3}{I_4}_u + \frac{4}{3}{I_6}_p + I_4^2(I_0^2-4 I_2) + I_0 I_4 (I_6 + 4 I_3)
  - \frac{14}{15} I_6^2 - \frac{172}{15} I_3 I_6\bigg) \\
+&\frac{{I_6}_r }{60 I_4^3} \bigg(I_4^2(-5I_0^2 - 12 I_1 + 48 I_2) - 5 I_0 I_4 I_6-20I_6^2\bigg) - \frac{3}{4I_4^2}\bigg(-\frac{2}{3}{I_4}_p (I_0 I_4 + I_6))+ \frac{2}{3}{I_6}_q (I_0 I_4 - \frac{8}{3} I_6)\bigg),\\
I_{10} &= \frac{ I_4 (36I_3 -18 I_6)+30{I_4}_p-5I_0{I_6}_r-45I_0I_4^2}{60 I_4}.
\end{aligned}
\end{equation}
This leads to the following constant structure equations when evaluated for the canonical forms \begin{equation}\label{both}u^{(4)}=\frac{5}{3}\frac{r^2}{q} ~\text{ or} ~ u^{(4)}=\frac{4}{3}\frac{r^2}{q}.\end{equation}
\begin{equation}\label{4/3}
\begin{aligned}
 d \theta^1&=-2\,\theta^1 \wedge \pi^{\prime }_1-\theta^2 \wedge \theta^5, \\
 d \theta^2&=c_1\,\theta^1\wedge \theta^4+c_2\,\theta^2\wedge \theta^3-\,\theta^2 \wedge \pi^{\prime }_1-\theta^3 \wedge \theta^5, \\
 d \theta^3&=c_3\,\theta^2 \wedge \theta^4-\theta^4 \wedge \theta^5, \\
 d \theta^4&=c_4\,\theta^3 \wedge \theta^4+\,\theta^4 \wedge \pi^{\prime }_1 ,\\
 d \theta^5&=c_5\,\theta^1 \wedge \theta^4+c_6\,\theta^3 \wedge \theta^5-\,\theta^5\wedge \pi'_1, \\
 d \pi^{\prime 1}&=c_7\,\theta^2 \wedge \theta^4+c_8\, \theta^4 \wedge \theta^5,
\end{aligned}
\end{equation}
where the constants $c_1,...,c_8$ are given in the following table for the two canonical forms.

\begin{table}[H]
\centering
\begin{tabular}{lcccccccc}
\toprule
Equation & $c_1$ & $c_2$ & $c_3$ & $c_4$ & $c_5$ & $c_6$ & $c_7$ & $c_8$ \\
\midrule
$u^{(4)}=\frac{5}{3}\frac{r^2}{q}$ 
& $\frac{3}{5}$ 
& $\frac{9}{10}$ 
& $\frac{9}{5}$ 
& $\frac{9}{10}$ 
& $-\frac{27}{25}$ 
& $\frac{9}{10}$ 
& $-\frac{27}{50}$ 
& $\frac{3}{10}$ \\

$u^{(4)}=\frac{4}{3}\frac{r^2}{q}$ 
& $-\frac{3}{4}$ 
& $\frac{9}{8}$ 
& $0$ 
& $-\frac{9}{8}$ 
& $0$ 
& $\frac{9}{8}$ 
& $0$ 
& $-\frac{3}{8}$ \\
\bottomrule
\end{tabular}
\caption{The constants $c_1,\dots,c_8$ correspond to the two fourth-order ODEs given in \eqref{both}.}
\end{table}

Theorem \ref{thm3} follows from \cite[Theorem 8.15, page 268]{Olver1995} and is given in the subsequent section.

 \section{Main Results}
 
 \begin{theorem}\label{thm0} A fourth-order ODE $u^{(4)}=f(x, u, p, q,r)$ is equivalent to the canonical form
\begin{equation*}
\begin{aligned}
u^{(4)}=\frac{6qr}{p}-\frac{6q^3}{p^2} 
\end{aligned}
\end{equation*}
with five symmetries under point transformation if and only if it belongs to the branch $I_4= I_{5 }= I_8=I_9=0,~I_6 \neq 0,~I_{10}\neq 0 $  and the exterior derivative of the coframe associated with this branch 

\begin{equation}\label{thm0I4}
 \begin{pmatrix}\\
 \theta^{1}\\
 \theta^{2}\\
 \theta^{3}\\ 
 \theta^4\\ 
 \theta^5\\\\
 \end{pmatrix}= 
  \left(\begin{array}{ccccc} 
J_6J_{10}& 0 & 0&0&0 \\ -I_{11} &J_6 & 0&0&0 \\ \frac{2}{3}\frac{{I_{11}^2}}{J_6J_{10}}&- \frac{4}{3}\frac{I_{11}}{J_{10}} & \frac{J_6}{J_{10}}&0&0\\ \frac{-2I_{11}^3-9J_6^3I_7}{9J_6^2J_{10}^2}&\frac{2I_{11}^2}{3J_6J_{10}^2}&-\frac{I_{11}}{J_{10}^2}&\frac{J_6}{J_{10}^2}&0\\
 \frac{6}{5}I_3J_{10}&0&0&0&J_{10} \end{array}\right) 
 \begin{pmatrix} \\
\omega^{1}\\
\omega^{2}\\
\omega^{3}\\ \omega^4\\
\omega^5\\\\
\end{pmatrix}, \end{equation} 
have identical constant structure equations for appropriate choice of $J_6,~J_{10}$, where\\
\begin{equation}\label{repeated}
\begin{aligned}
I_0&=-f_r,\quad I_1=-\frac{1}{6}\hat{D}_xI_0+\frac{1}{4}I_0^2,\quad I_2=-\frac{1}{2}\hat{D}_xI_0-f_q,\quad I_3=-\frac{1}{8} {I _0}_r I_ 0+\frac{1}{4} {I_ 1}_r-\frac{1}{12} {I_ 2}_r,\\
\quad I _4&=-\frac{1}{6} {I_ 0}_r,\quad I_5=-\frac{18}{5}I_0I_3-\frac{3}{5}{I_0}_p+\frac{2}{5}(I_2-3I_1)_q,\quad I_6=-\frac{36}{25}I_3^2-\frac{3}{5}I_0{I_3}_q-\frac{6}{5}{I_3}_p,\\
J_6^2&=I_6,~I_7=\frac{1}{2}I_0^3-\frac{3}{10}\left(\hat{D}_x(I_1+I_2)-f(I_1+I_2)_r\right)+\frac{6}{5}fI_3-\frac{1}{20}I_0(I_2-39I_1),\\
I_8&=\frac{7}{8}I_0^3-(\hat{D}_xI_2+{I_2}_rf)-f_p+6fI_3+I_0(-3I_1-\frac{1}{2}I_2),\quad I_9=-\frac{6}{5}{I_3}_q,\\
I_{10}&=\hat{D}_xI_7-f_u+I_0I_7+(I_1+I_2)(\frac{9}{100}(9I_1-I_2)-\frac{9}{40}I_0^2),\quad J_{10}^4=I_{10},\\
I_{11}&=-3 {J_6}_x-3p{J_6}_u-3q{J_6}_p.
\end{aligned}
 \end{equation}

\end{theorem}

\begin{theorem}\label{thm1} A fourth-order ODE $u^{(4)}=f(x, u, p, q,r)$ is equivalent to one of the canonical forms
\begin{equation*}
\begin{aligned}
\text{(i)} ~u^{(4)}&=e^{r} ~\text{(ii)} ~u^{(4)}=r^{\frac{3}{2}},~\text{(iii)}~u^{(4)}=r^{\frac{4}{3}},~\text{(iv)}~u^{(4)}=r^{\frac{b-3}{b-2}},b \neq 1,2,3\\
\text{(v)}~u^{(4)}&=\frac{6qr}{p}-\frac{6q^3}{p^2}+K\frac{(3q^2-2pr)^\frac{3}{2}}{p^2},K\neq 0~\text{(vi)}~u^{(4)}=r^{\frac{1-3b}{1-2b}},b \neq 0,~\frac{1}{3},~\frac{1}{2},~1 
\end{aligned}
\end{equation*}
with five symmetries under point transformation if and only if it belongs to the branch $I_4\neq 0, I_{5 }\neq 0$ and the exterior derivative of the coframe associated with this branch 

\begin{equation}\label{2.19}
 \begin{pmatrix}\\
 \theta^{1}\\
 \theta^{2}\\
 \theta^{3}\\ 
 \theta^4\\ 
 \theta^5\\\\
 \end{pmatrix}= 
 \left(\begin{array}{ccccc} 
 \frac{I_4}{I_5^2} & 0 & 0&0&0 \\ \frac{I_{6}}{I_5} & \frac{I_4}{I_5} & 0&0&0 \\ \frac{2}{3}\frac{{I_6^2}}{I_4}& \frac{4}{3}I_6 & I_4&0&0\\ -I_4I_5I_7+\frac{2}{9}\frac{I_5I_6^3}{I_4^2}&\frac{2}{3}\frac{I_5I_6^2}{I_4}&I_5I_6&I_4I_5&0\\
 -\frac{3}{5}\frac{I_6}{I_5}+\frac{6}{5}\frac{I_3}{I_5}&0&0&0&\frac{1}{I_5} \end{array}\right) 
 \begin{pmatrix} \\
\omega^{1}\\
\omega^{2}\\
\omega^{3}\\ \omega^4\\
\omega^5\\\\
\end{pmatrix}, \end{equation} 

 have identical constant structure equations, where\\
\begin{equation}\label{repeated}
\begin{aligned}
I_0&=-f_r,\quad I_1=\frac{1}{4}I_0^2-\frac{1}{6}\hat{D}_xI_0,\quad I_2=-\frac{1}{2}\hat{D}_xI_0-f_q,\quad I_3=-\frac{1}{8} {I _0}_r I_ 0+\frac{1}{4} {I_ 1}_r-\frac{1}{12} {I_ 2}_r,\\
\quad I _4&=-\frac{1}{6} {I_ 0}_r,\quad I_5=- \frac{1}{2}\frac{{I_4}_r}{I_4^2},\quad I_6=-\frac{3}{2} I _0 I_ 4-{I_ 2}_r-6 I _3,\\
I_7&=\frac{1}{10}f\left(9I_0I_4+(-I_2+3I_1)_r-12I_3\right)+\frac{1}{2}I_0(I_0^2+\frac{1}{10}I_2)-\frac{3}{10}\hat{D}_x(I_1+I_2)-\frac{39}{20}I_0I_1.
\end{aligned}
 \end{equation}

Moreover, the parameter $b$ in \text{(iv)} can be deduced through the invariant relation $\frac{T^2_{24}}{T^5_{45}}=-b,$ in \text{(vi)} by $\frac{T^5_{45}}{T^2_{24}}=-b$ and  the parameter $K$ can be found using the invariant relation $T^4_{25}=\frac{3K^2-4}{8K^2}.$
\end{theorem}

\begin{theorem}\label{thm2}
A fourth-order ODE $u^{(4)}=f(x, u, p, q,r)$ is equivalent to one of the canonical forms
\begin{equation*}
\begin{aligned}
\text{(i)} ~u^{(4)}&=r^2 ~\text{(ii)} ~u^{(4)}=K\frac{r^2}{q},~ K\neq 0, \frac{4}{3},\frac{5}{3},~\text{(iii)}~u^{(4)}=\frac{5}{3}\frac{r^2}{q}+q^{\frac{5}{3}},~\text{(iv)}~u^{(4)}= \frac{4}{3}\frac{r^2}{q}+q^{\frac{7}{3}},
\end{aligned}
\end{equation*}
with five symmetries under point transformation if and only if it belongs to the branch $I_4\neq 0, I_{5 }= 0,~I_8 \neq 0$  and the exterior derivative of the coframe associated with this branch 

\begin{equation}\label{2.19}
 \begin{pmatrix}\\
 \theta^{1}\\
 \theta^{2}\\
 \theta^{3}\\ 
 \theta^4\\ 
 \theta^5\\\\
 \end{pmatrix}= 
 \left(\begin{array}{ccccc}
I_4J_8^2 & 0 & 0 & 0 & 0 \\[0.8em]
I_6J_8 & I_4J_8 & 0 & 0 & 0 \\[0.8em]
\frac{2}{3} \frac{I_6^2}{I_4} & \frac{4}{3}I_6 &I_4 & 0 & 0 \\[0.8em]
-\frac{I_4I_7}{J_8}+\frac{2}{9}\frac{I_6^3}{I_4^2J_8} &\frac{2}{3}\frac{I_6^2}{I_4J_8} &\frac{I_6}{J_8} & \frac{I_4}{J_8}&0 \\[0.8em]
-\frac{3}{5}I_6J_8+\frac{6}{5}I_3J_8 &0 &0 & 0 &J_8
\end{array}\right)
 \begin{pmatrix} \\
\omega^{1}\\
\omega^{2}\\
\omega^{3}\\ \omega^4\\
\omega^5\\\\
\end{pmatrix}, \end{equation} 
have identical constant structure equations for appropriate choice for $ J_8$,\\ where
 $I_0 ,....,I_7$ are given by equation \eqref{repeated} and 
\begin{equation}\label{repeated2}
\begin{aligned}
I_8&=\frac{I_6}{I_4^2}\hat{D}_x(I_4-{I_4}_rf)-\frac{1}{I_4}\hat{D}_xI_6+\frac{3}{4}I_0^2-\frac{1}{2}\frac{I_0I_6}{I_4}+\frac{3}{10}I_2+\frac{1}{3}\frac{I_6^2}{I_4^2}-\frac{27}{10}I_1,\, J_8^2=I_8.
\end{aligned}
 \end{equation}

Furthermore, the parameter $K$ can be deduced using the invariant relation $T^3_{13}=\frac{2K-3}{K}.$

\end{theorem}

\begin{theorem}\label{thm3}
 A fourth-order ODE $u^{(4)}=f(x, u, p, q,r)$ is equivalent to one of the canonical forms
\begin{equation*}
\begin{aligned}
\text{(i)}~u^{(4)}=\frac{5}{3}\frac{r^{2}}{q},~\text{(ii)}~u^{(4)}=\frac{4}{3}\frac{r^2}{q}\end{aligned}
\end{equation*}
with six symmetries under point transformation if and only if it belongs to the branch $I_4\neq 0,~I_5 = 0,~I_{8}=0$ and the exterior derivative of the coframe associated with this branch is
\begin{equation}\label{2.9}
\left(\begin{array}{l}
\theta^1 \\
\theta^2 \\
\theta^3 \\
\theta^4\\
\theta^5\\
\pi_1'
\end{array}\right)=
\left(\begin{array}{cccccc}
a_{13}^2I_4 & 0 & 0&0&0&0  \\
a_{13} I_{6}& a_{13}I_4 &0&0 & 0 &0\\
\frac{2}{3}\frac{I_6^2}{I_4}&\frac{4}{3}I_6&I_4&0&0&0\\
-\frac{I_4I_7}{a_{13}}+\frac{2}{9}\frac{I_6^3}{a_{13}I_4^2}&\frac{2}{3}\frac{I_6^2}{a_{13}I_4}&\frac{I_{6}}{a_{13}}&\frac{I_4}{a_{13}}&0&0\\
-\frac{3}{5}a_{13}I_6+\frac{6}{5}a_{13}I_3&0&0&0&a_{13}&0\\
I_9&I_{10}&\frac{-9I_4^2-{I_6}_r}{6I_4}&0&\frac{3I_0I_4-4I_6}{6I_4}&\frac{1}{a_{13}}
\end{array}\right)
\left(\begin{array}{c}
\omega^1 \\
\omega^2 \\
\omega^3\\
\omega^4\\
\omega^5\\
d a_{13}
\end{array}\right),
\end{equation}
where $ I_0,....,I_8$ are given by equations \eqref{repeated},  \eqref{repeated2} and
\begin{equation}\label{I9I10}
\begin{aligned}
I_9 &=  -\frac{3}{4I_4} \bigg(-\frac{2}{3}{I_0}_p I_4 - \frac{4}{3}{I_4}_u + \frac{4}{3}{I_6}_p + I_4^2(I_0^2-4 I_2) + I_0 I_4 (I_6 + 4 I_3)
  - \frac{14}{15} I_6^2 - \frac{172}{15} I_3 I_6\bigg) \\
+&\frac{{I_6}_r }{60 I_4^3} \bigg(I_4^2(-5I_0^2 - 12 I_1 + 48 I_2) - 5 I_0 I_4 I_6-20I_6^2\bigg) - \frac{3}{4I_4^2}\bigg(-\frac{2}{3}{I_4}_p (I_0 I_4 + I_6))+ \frac{2}{3}{I_6}_q (I_0 I_4 - \frac{8}{3} I_6)\bigg),\\
I_{10} &= \frac{ I_4 (36I_3 -18 I_6)+30{I_4}_p-5I_0{I_6}_r-45I_0I_4^2}{60 I_4}.
\end{aligned}
\end{equation}

We have the following constant structure equations once evaluated at the canonical forms \begin{equation}\label{both2}u^{(4)}=\frac{5}{3}\frac{r^2}{q}~ \text{or}~ u^{(4)}=\frac{4}{3}\frac{r^2}{q}.\end{equation}
\begin{equation}\label{4/3c}
\begin{aligned}
 d \theta^1&=-2\,\theta^1 \wedge \pi^{\prime }_1-\theta^2 \wedge \theta^5, \\
 d \theta^2&=c_1\,\theta^1\wedge \theta^4+c_2\,\theta^2\wedge \theta^3-\,\theta^2 \wedge \pi^{\prime }_1-\theta^3 \wedge \theta^5, \\
 d \theta^3&=c_3\,\theta^2 \wedge \theta^4-\theta^4 \wedge \theta^5, \\
 d \theta^4&=c_4\,\theta^3 \wedge \theta^4+\,\theta^4 \wedge \pi^{\prime }_1 ,\\
 d \theta^5&=c_5\,\theta^1 \wedge \theta^4+c_6\,\theta^3 \wedge \theta^5-\,\theta^5\wedge \pi'_1, \\
 d \pi^{\prime 1}&=c_7\,\theta^2 \wedge \theta^4+c_8\, \theta^4 \wedge \theta^5,
\end{aligned}
\end{equation}
where the constants for the two canonical forms are given in the following table.
\begin{table}[H]
\centering
\begin{tabular}{lcccccccc}
\toprule
Equation & $c_1$ & $c_2$ & $c_3$ & $c_4$ & $c_5$ & $c_6$ & $c_7$ & $c_8$ \\
\midrule
$u^{(4)}=\frac{5}{3}\frac{r^2}{q}$ 
& $\frac{3}{5}$ 
& $\frac{9}{10}$ 
& $\frac{9}{5}$ 
& $\frac{9}{10}$ 
& $-\frac{27}{25}$ 
& $\frac{9}{10}$ 
& $-\frac{27}{50}$ 
& $\frac{3}{10}$ \\

$u^{(4)}=\frac{4}{3}\frac{r^2}{q}$ 
& $-\frac{3}{4}$ 
& $\frac{9}{8}$ 
& $0$ 
& $-\frac{9}{8}$ 
& $0$ 
& $\frac{9}{8}$ 
& $0$ 
& $-\frac{3}{8}$ \\
\bottomrule
\end{tabular}
\caption{The constants $c_1,\dots,c_8$ correspond to the two fourth-order ODEs given in \eqref{both2}.}
\end{table}

\end{theorem}

\section{Determination of point transformation via invariant coframes}
In this section, we present a method for constructing the point transformation between two equivalent fourth-order ODEs belonging to the canonical forms listed in Table 1. This approach is based on the invariant coframes established in the previous section, together with the following propositions.
\begin{proposition}\label{prop1} Assume that the fourth-order ODEs

\begin{equation}\label{equi}
u^{(4)}=f\left(x, u, u^{\prime}, u^{\prime \prime}, u^{\prime \prime \prime}\right), \quad \bar{u}^{(4)}=\bar{f}\left(\bar{x}, \bar{u}, \bar{u}^{\prime}, \bar{u}^{\prime \prime}, \bar{u}^{\prime \prime \prime}\right),
\end{equation}
are equivalent under the point transformation

\begin{equation}\label{point}
\bar{x}=\varphi(x, u), \quad\bar{u}=\psi(x, u),\quad \phi_x\psi_u-\phi_u\psi_x \neq 0.
\end{equation} 
Given an invariant five-dimensional coframe on the space $M$ such that

\begin{equation}\label{eq}
\Phi^*\left(\begin{array}{ccccc}
\bar{a}_1 & 0 & 0 & 0 &0\\
\bar{a}_2 & \bar{a}_3 & 0 & 0&0 \\
\bar{a}_4 & \bar{a}_5 & \bar{a}_6 & 0&0 \\
\bar{a}_7 & \bar{a}_8 &  \bar{a}_9&\bar{a}_{10}&0\\
\bar{a}_{11}& 0&0&0&\bar{a}_{13}
\end{array}\right)\left(\begin{array}{c}
 \bar{\omega}^1\\
 \bar{\omega}^2\\
 \bar{\omega}^3\\
 \bar{\omega}^4\\
 \bar{\omega}^5
\end{array}\right)=\left(\begin{array}{ccccc}
a_1 & 0 & 0 & 0 &0\\
a_2 & a_3 & 0 & 0 &0\\
a_4 & a_5 & a_6 & 0 &0\\
a_7 & a_8 &  a_9&a_10&0\\
a_{11}&0&0&0&a_{13}
\end{array}\right)\left(\begin{array}{c}
\omega^1 \\
\omega^2 \\
\omega^3 \\
\omega^4\\
\omega^5
\end{array}\right),
\end{equation}
for some functions $a_i(x, u, p, q, r), i=1,2 \ldots 11, 13$ with $\Phi^*$ representing the pullback associated with the mapping determined by the third prolongation of the point transformation \eqref{point} and $$\left(\begin{array}{l}
\omega^1,\omega^2,\omega^3,\omega^4,\omega^5
\end{array}\right)^T
=\Omega
\left(\begin{array}{c}
du-pdx,
dp-qdx,
dq-rdx,
dr-fdx,
dx,
\end{array}\right)^T$$ is an adapted coframe for some invertible matrix  $\Omega$. Then the following system can be used to construct the point transformation between the fourth-order ODEs \eqref{equi}.
\begin{equation}\label{pdesys}
\begin{aligned}
&\hat{D}_x\xi=\bar{f}b_{13},\quad \xi_u=\bar{f}b_{11}+b_7,\quad \xi_p=b_8,\quad \xi_q=b_9,\quad\xi_r=b_{10},\\
&\hat{D}_x\eta=\xi b_{13},\quad\eta_u=\xi b_{11}+b_4,\quad \eta_{p}=b_5,\quad \eta_q=b_6,\\
&\hat{D}_x g=\eta b_{13},\quad g_u=\eta b_{11}+b_2,\quad g_p=b_3,\\
&\hat{D}_x\psi=gb_{13},\quad \psi_u=gb_{11}+b_1, \\
&\hat{D}_x\phi=b_{13},\quad\phi_u=b_{11},
\end{aligned}
\end{equation}
where 
\begin{equation*}
g=\frac{\hat{D}_x\psi}{\hat{D}_x \phi},\quad \eta=\frac{\hat{D}_xg}{\hat{D}_x \phi}, \quad \xi=\frac{\hat{D}_x \eta}{\hat{D}_x \phi},\quad \hat{D}_x=\frac{\partial}{\partial x}+p\frac{\partial}{\partial u}+q\frac{\partial}{\partial p}+r\frac{\partial }{\partial q}+f\frac{\partial}{\partial r},\\
\end{equation*}
and 

\begin{equation}\label{eqq}
\left(\begin{array}{ccccc}
b_1 & 0 & 0 & 0&0 \\
b_2 & b_3 & 0 & 0 &0\\
b_4 & b_5 & b_6 & 0&0 \\
b_7 & b_8 &  b_9&b_{10}&0\\
b_{11}&b_{12}&0&0&b_{13}
\end{array}\right)=\bar{\Omega}^{-1}\setlength{\arraycolsep}{1pt}\left(\begin{array}{ccccc}
\bar{a}_1 & 0 & 0 & 0&0 \\
\bar{a}_2 & \bar{a}_3 & 0 & 0 &0\\
\bar{a}_4 & \bar{a}_5 & \bar{a}_6 & 0 &0\\
\bar{a}_7 & \bar{a}_8 & \bar{a}_9&\bar{ a}_{10}&0\\
\bar{a}_{11}&0&0&0&\bar{a}_{13}
\end{array}\right)^{-1}\left(\begin{array}{ccccc}
a_1 & 0 & 0 & 0 &0\\
a_2 & a_3 & 0 & 0&0 \\
a_4 & a_5 & a_6 & 0 &0\\
a_7 & a_8 &  a_9&a_{10}&0\\
a_{11}& 0&0&0&a_{13}
\end{array}\right)\Omega.
\end{equation}

\begin{proof} Equation \eqref{eq} can be represented in terms of \eqref{eqq} as

\begin{equation}\label{pullback}
\left(\begin{array}{c}
d \bar{u}-\bar{p} d \bar{x} \\
d \bar{p}-\bar{q} d \bar{x} \\
d \bar{q}-\bar{r} d \bar{x} \\
d \bar{r}-\bar{f} d \bar{x}\\
d \bar{x}
\end{array}\right)=\left(\begin{array}{ccccc}
b_1 & 0 & 0 & 0 &0\\
b_2 & b_3 & 0 & 0 &0\\
b_4 & b_5 & b_6 & 0 &0\\
b_7 & b_8 & b_9& b_{10}&0\\
b_{11}& 0&0&0&b_{13}
\end{array}\right)\left(\begin{array}{c}
d u-p d x \\
d p-q d x \\
d q-r d x \\
d r-f d x\\
d x
\end{array}\right).
\end{equation}

On the other hand, the pullback of the left-hand side of the equation \eqref{pullback} can be calculated which renders system \eqref{pdesys}.

This completes the proof.
\end{proof}
\end{proposition}

\begin{proposition}\label{prop2}
 Assume that the fourth-order ODEs

\begin{equation}\label{eq2}
u^{(4)}=f\left(x, u, u^{\prime}, u^{\prime \prime}, u^{\prime \prime \prime}\right), \quad \bar{u}^{(4)}=\frac{4}{3} \frac{\bar{r}^2}{\bar{q}},
\end{equation}
are equivalent under the point transformation

\begin{equation}\label{point2}
\bar{x}=\varphi(x, u), ~\bar{u}=\psi(x, u), \quad \phi_x\psi_u-\phi_u \psi_x \neq 0.
\end{equation}

Hence, the point transformation \eqref{point2} can be constructed using the following system between the fourth-order ODEs \eqref{eq2}
\begin{equation}\label{sys2}
\begin{aligned}
\hat{D}_xa_{13}&=-\frac{b_{17}}{b_{18}},\quad {a_{13}}_u=-\frac{b_{14}}{b_{18}},\quad {a_{13}}_p=-\frac{b_{15}}{b_{18}},\quad {a_{13}}_q=-\frac{b_{16}}{b_{18}},\\
\hat{D}_x\xi&=\bar{f}b_{13},\quad \xi_u=\bar{f}b_{11}+b_7,\quad \xi_p=b_8,\quad \xi_q=b_9, \quad \xi_r=b_{10},\\
\hat{D}_x\eta &=\xi b_{13},\quad \eta_u=\xi b_{11}+b_4,\quad \eta_p=b_5,\quad \eta_q=b_6,\\
\hat{D}_x g &=\eta b_{13},\quad g_u=\eta b_{11}+b_2,\quad g_p=b_3,\\
\hat{D}_x\psi&=gb_{13},\quad \psi_u=gb_{11}+b_1,\\
\hat{D}_x\phi&=b_{13},\quad \phi_u=b_{11},
\end{aligned}
\end{equation}
where $a_{13}(x,u,p,q)$ is an auxiliary function and $g=\frac{\hat{D}_x\psi}{\hat{D}_x \phi},~ \eta=\frac{\hat{D}_x g}{\hat{D}_x\phi},\\ \xi=\frac{\hat{D}_x\eta}{\hat{D}_x \phi},$~  $\bar{f}=\frac{4}{3}\frac{\xi^2}{\eta}$ and 

\begin{equation}\label{b2}
\begin{aligned}
\begin{pmatrix}
b_1 & 0 & 0 & 0 & 0 & 0 \\
b_2 & b_3 & 0 & 0 & 0 & 0 \\
b_4 & b_5 & b_6 & 0 & 0 & 0 \\
b_7 & b_8 & b_9 & b_{10} & 0 & 0 \\
b_{11} & 0 & 0 & 0 & b_{13} & 0 \\
b_{14} & b_{15} & b_{16} & 0 & b_{17} & b_{18}
\end{pmatrix}
=(\bar{A}\bar{H})^{-1}AH
\end{aligned}
\end{equation}
where\setlength{\arraycolsep}{1pt}
\begin{equation}\label{H}
A=\left(\begin{array}{cccccc}
a_{13}^2I_4 & 0 & 0&0&0&0  \\
a_{13} I_{6}& a_{13}I_4 &0&0 & 0 &0\\
\frac{2}{3}\frac{I_6^2}{I_4}&\frac{4}{3}I_6&I_4&0&0&0\\
-\frac{I_4I_7}{a_{13}}+\frac{2}{9}\frac{I_6^3}{a_{13}I_4^2}&\frac{2}{3}\frac{I_6^2}{a_{13}I_4}&\frac{I_{6}}{a_{13}}&\frac{I_4}{a_{13}}&0&0\\
-\frac{3}{5}a_{13}I_6+\frac{6}{5}a_{13}I_3&0&0&a_{13}&0&0\\
I_9&I_{10}&\frac{-9I_4^2-{I_6}_r}{6I_4}&0&\frac{3I_0I_4-4I_6}{6I_4}&\frac{1}{a_{13}}
\end{array}\right),
\quad
H=\left(
\begin{array}{cc}
\Omega&0\\
0&1
\end{array}
\right),
\end{equation}
$\Omega$ is given in \eqref{2.2}, $\bar{A},~\bar{H}$ obtained by evaluating the matrices $A,~H$ for the canonical form $\bar{u}^{(4)}=\frac{4}{3}\frac{\bar{r}^2}{\bar{q}} $ and $I_0,... ,I_{10}$ are given in  \eqref{repeated}, \eqref{repeated2}, \eqref{I9I10}.

\end{proposition}
\begin{proof}
The equivalence of the fourth-order ODEs \eqref{eq2} under point transformation can be checked by using the invariant coframe \eqref{2.99} on the space $M^{(1)}=M \times G_4$ such that
\begin{equation}\label{eq22}
\begin{aligned}
&\Phi^*\!
\bar{A}\bar{H}
\left(
\begin{array}{c}
d\bar{u}-\bar{p}\,d\bar{x}\\
d\bar{p}-\bar{q}\,d\bar{x}\\
d\bar{q}-\bar{r}\,d\bar{x}\\
d\bar{r}-\bar{f}\,d\bar{x}\\
d\bar{x}\\
d\bar{a}_{13}
\end{array}
\right)
=
A H
\left(
\begin{array}{c}
du-p\,dx\\
dp-q\,dx\\
dq-r\,dx\\
dr-f\,dx\\
dx\\
da_{13}
\end{array}
\right).
\end{aligned}
\end{equation}
where $\Phi^{*}$ denotes the pullback induced by the third prolongation of the point transformation \eqref{point2}.
Substitute  $\bar{a}_{13}=1$ in the left hand side of \eqref{eq22}. Then \eqref{eq22} can be represented in terms of \eqref{b2} as 
\begin{equation}\label{b3}
\left(\begin{array}{c}
d \bar{u}-\bar{p} d \bar{x} \\
d \bar{p}-\bar{q} d \bar{x} \\
d \bar{q}-\bar{r} d \bar{x} \\
d \bar{r}-\bar{f} d \bar{x}\\
d \bar{x}\\
0
\end{array}\right)=(\bar{A}\bar{H})^{-1}AH \left(\begin{array}{c}
d u-p d x \\
d p-q d x \\
d q-r d x \\
d r-f d x\\
d x\\
d a_{13}
\end{array}\right)=\setlength{\arraycolsep}{0.5pt}
\left(\begin{array}{cccccc}
b_1 & 0 & 0 & 0 &0&0\\
b_2 & b_3 & 0 & 0 &0&0\\
b_4 & b_5 & b_6 & 0 &0&0\\
b_7 & b_8 & b_9& b_{10}&0&0\\
b_{11}& 0&0&0&b_{13}&0\\
b_{14}&b_{15}&b_{16}&0&b_{17}&b_{18}
\end{array}\right)\left(\begin{array}{c}
d u-p d x \\
d p-q d x \\
d q-r d x \\
d r-f d x\\
d x\\
d a_{13}
\end{array}\right)
\end{equation}
By doing the pullback of the left hand side of \eqref{b3}, we deduce equations \eqref{sys2}.
\end{proof}

\begin{remark}
 If a fourth-order ODE is equivalent to $\bar{u}^{(4)}=\frac{5}{3}\frac{\bar{r}^2}{\bar{q}},$  we can similarly, construct the point transformation  by using Proposition \ref{prop2}, where $\bar{A},~\bar{H}$ are obtained by evaluating the matrices $A,~H$ for the canonical form $\bar{u}^{(4
 )}=\frac{5}{3}\frac{\bar{r}^2}{\bar{q}}.$\end{remark}

\begin{remark}
The explicit values for the matrices $\bar{\Omega},~\bar{A},$ used in Proposition \ref{prop2} after incorporating $\bar{a}_{13}=1$ are given for each canonical form as follows \setlength{\arraycolsep}{0.5pt}
\begin{equation}
\bar{\Omega}=\left(
\begin{array}{ccccc}
1&0&0&0&0\\
\frac{n_1\xi}{\eta}&1&0&0&0\\
\frac{n_2\xi^2}{\eta^2}&\frac{n_3\xi}{\eta}&1&0&0\\
0&\frac{n_4\xi^2}{\eta^2}&0&1&0\\
0&0&0&0&1\\
\end{array}
\right),\quad \bar{A}=\left(\begin{array}{cccccc}
\frac{m_1}{\eta} & 0 & 0&0&0&0  \\
\frac{m_2\xi}{\eta^2}&\frac{m_3}{\eta} &0&0 & 0 &0\\
\frac{m_4\xi^2}{\eta^3}&\frac{m_5\xi}{\eta^2}&\frac{m_6}{\eta}&0&0&0\\
\frac{m_7\xi^3}{\eta^4}&\frac{m_8\xi^2}{\eta^3}&\frac{m_9\xi}{\eta^2}&\frac{m_{10}}{\eta}&0&0\\
\frac{m_{11}\xi}{\eta^2}&0&0&0&a_{13}&0\\
\frac{m_{12}\xi^2}{\eta^3}&\frac{m_{13}\xi}{\eta^2}&\frac{m_{14}}{\eta}&0&\frac{m_{15}\xi}{\eta}&1
\end{array}\right),
\end{equation}
\end{remark}
where the constants $m_1,\dots ,m_{15}$ and $n_1,\dots ,n_4$ are given in the following table
\begin{table}[H]
\centering
\scriptsize
\setlength{\tabcolsep}{3pt}
\setlength{\heavyrulewidth}{1.2pt}
\setlength{\lightrulewidth}{0.8pt}{
\boldmath
\resizebox{\textwidth}{!}{
\begin{tabular}{lccccccccccccccc|cccc}
\toprule
\textbf{Equation} 
& $\mathbf{m_1}$ & $\mathbf{m_2}$ & $\mathbf{m_3}$ & $\mathbf{m_4}$ & $\mathbf{m_5}$ 
& $\mathbf{m_6}$ & $\mathbf{m_7}$ & $\mathbf{m_8}$ 
& $\mathbf{m_9}$ & $\mathbf{m_{10}}$ & $\mathbf{m_{11}}$ & $\mathbf{m_{12}}$ 
& $\mathbf{m_{13}}$ & $\mathbf{m_{14}}$ & $\mathbf{m_{15}}$ 
& $\mathbf{n_1}$ & $\mathbf{n_2}$ & $\mathbf{n_3}$ & $\mathbf{n_4}$ \\
\midrule

$\mathbf{u^{(4)}=\frac{4}{3}\frac{r^2}{q}}$ 
& $\frac{4}{9}$ 
& $-\frac{4}{9}$ 
& $\frac{4}{9}$ 
& $\frac{8}{27}$ 
& $-\frac{16}{27}$ 
& $\frac{4}{9}$ 
& $\frac{8}{81}$ 
& $\frac{8}{27}$ 
& $-\frac{4}{9}$
& $\frac{4}{9}$
& $0$
& $-\frac{1}{3}$
& $\frac{2}{3}$
& $-\frac{1}{2}$
& $-\frac{2}{3}$
& $\frac{4}{3}$
& $\frac{10}{9}$
& $\frac{4}{3}$
& $\frac{2}{3}$ \\

$\mathbf{u^{(4)}=\frac{5}{3}\frac{r^2}{q}}$ 
& $\frac{5}{9}$ 
& $-\frac{10}{9}$ 
& $\frac{5}{9}$ 
& $\frac{40}{27}$ 
& $-\frac{40}{27}$ 
& $\frac{5}{9}$ 
& $-\frac{40}{81}$ 
& $\frac{40}{27}$ 
& $-\frac{10}{9}$
& $\frac{5}{9}$
& $\frac{1}{3}$
& $-\frac{7}{9}$
& $1$
& $-\frac{1}{2}$
& $-\frac{1}{3}$
& $\frac{5}{3}$
& $\frac{16}{9}$
& $\frac{5}{3}$
& $1$ \\

\bottomrule
\end{tabular}}}
\normalsize
\caption{Values of the constants $m_1,\dots,m_{15}$ and $n_1,\dots,n_4$ for the two fourth-order ODEs given in \eqref{both2}.}
\end{table}
\begin{example}
Consider the class of non-linear fourth-order ODE

\begin{equation}\label{example0}
u^{(4)}=\frac{6u_{xx}\large\left((1+u_x)u_{xxx}-u_{xx}^2\large\right)}{(1+u_x)^2} .
\end{equation}

It can be checked that the function $f(x, u, p, q, r)=\frac{6q((1+p)r-q^2)}{(1+p)^2}$ satisfies the conditions of Theorem \ref{thm0}. Thus it is equivalent to the canonical form

\begin{equation}\label{example00}
\bar{u}^{(4)}=\frac{6\bar{u}_{\bar{x}\bar{x}}\bar{u}_{\bar{x}\bar{x}\bar{x}}}{\bar{u}_{\bar{x}}}-\frac{6\bar{u}_{\bar{x}\bar{x}}^3}{\bar{u}_{\bar{x}}}.
\end{equation} 
It can be checked that they have five symmetries, under point transformation and identical constant structure equations given by

\begin{equation}
\begin{aligned}
 d \theta^1&=\frac{1}{4}\sqrt{6}\,\theta^1 \wedge \theta^2-\frac{1}{4}i\sqrt{6}\,\theta^1 \wedge \theta^4-\theta^2 \wedge \theta^5, \\
 d \theta^2&=\frac{1}{2}\sqrt{6}\,\theta^1\wedge \theta^3+-i\,\theta^1\wedge \theta^5-\theta^3 \wedge \theta^5, \\
 d \theta^3&=\frac{1}{3}i\sqrt{6}\,\theta^1 \wedge \theta^2+\frac{7}{12}\sqrt{6}\,\theta^2\wedge \theta^3-\frac{4}{3}i\,\theta^2\wedge \theta^5+\frac{1}{4}i\sqrt{6}\,\theta^3 \wedge \theta^4-\theta^4 \wedge \theta^5, \\
 d \theta^4&=\frac{1}{2}i\sqrt{6}\,\theta^1 \wedge \theta^3+\theta^1 \wedge \theta^5-i\,\theta^3\wedge \theta^5 ,\\
 d \theta^5&=\theta^1 \wedge \theta^2-\frac{1}{12}\sqrt{6}\,\theta^2 \wedge \theta^5+\frac{1}{4} \sqrt{6}i\,\theta^4\wedge \theta^5, \\
\end{aligned}
\end{equation}
by choosing $J_6=-\frac{\sqrt{5}}{6(1+p)},~J_{10}=-\frac{\sqrt{-60ir(1+p)+90iq^2}}{10(1+p)},~\bar{J_6}=-\frac{1}{5}\frac{\sqrt{6}}{g},~\bar{J}_{10}=-\frac{\sqrt{30}i\large\left(-(2g\xi-3\eta^2)^2\large\right)^{\frac{1}{4}}}{10g}.$

The construction of the point transformation \eqref{point} requires the computation of the matrix entries \eqref{eqq} by using \eqref{example0} and \eqref{example00} as below 
$$
\left(\begin{array}{ccccc}
b_1 & 0 & 0 & 0& 0\\
b_2 & b_3 & 0 & 0& 0 \\
b_4 & b_5 & b_6 & 0&0 \\
b_7 & b_8 & b_9& b_{10}&0\\
b_{11}& 0&0&0&b_{13}
\end{array}\right)=\left(\begin{array}{ccccc}
\frac{Bg^2}{C(1+p)^2}& 0 & 0 & 0 & 0\\
0&\frac{g}{1+p} & 0 & 0&0 \\
0 & \frac{B\eta-Cq}{(1+p)B} &\frac{C}{B}& 0 &0\\
0& \frac{3}{2}\frac{(B\eta-Cq)^2}{B^2(1+p)g} & \frac{3C( B\eta-Cq)}{B^2g}&\frac{C^2(1+p)}{gB^2}&0\\
0&0&0&0&\frac{gB}{C(1+p)}
\end{array}\right),
$$
where $ B^2=2pr-3q^2+2r,\quad C^2=2g\xi-3\eta^2.$
Solving the system \eqref{pdesys} given in Proposition \ref{prop1} results in a point transformation

$$
\bar{x}=\frac{1}{x}, \quad \bar{u}=x+u .
$$
\end{example}

\begin{example}
Consider the class of non-linear fourth-order ODE

\begin{equation}\label{example1}
u^{(4)}=\frac{(xu_{xxx}+3u_{xx})^\frac{4}{3}-4u_{xxx}}{x} .
\end{equation}

It can be checked that the function $f(x, u, p, q, r)=\frac{(xr+3q)^\frac{4}{3}-4r}{x}$ satisfies the conditions of Theorem \ref{thm1} with identical constant structure equations to the canonical form

\begin{equation}\label{example11}
\bar{u}^{(4)}=\bar{u}_{\bar{x}\bar{x}\bar{x}}^\frac{4}{3},
\end{equation} 
given by

\begin{equation}
\begin{aligned}
 d \theta^1&=3\,\theta^1 \wedge \theta^5-\theta^2 \wedge \theta^5, \\
 d \theta^2&=-3\,\theta^1\wedge \theta^4+\,\theta^2\wedge \theta^4+\frac{3}{2}\,\theta^2 \wedge \theta^5-\theta^3 \wedge \theta^5, \\
 d \theta^3&=-3\,\theta^2 \wedge \theta^4+2\,\theta^3 \wedge \theta^4-\theta^4\wedge \theta^5, \\
 d \theta^4&=-\frac{3}{2}\,\theta^4 \wedge \theta^5,\\
 d \theta^5&=\theta^4 \wedge \theta^5. \\
\end{aligned}
\end{equation}

 So both are equivalent with five symmetries under point transformation.
The construction of the point transformation \eqref{point} requires the computation of the matrix entries \eqref{eqq} by using \eqref{example1} and \eqref{example11} as below 
$$
\left(\begin{array}{ccccc}
b_1 & 0 & 0 & 0& 0\\
b_2 & b_3 & 0 & 0& 0 \\
b_4 & b_5 & b_6 & 0&0 \\
b_7 & b_8 & b_9& b_{10}&0\\
b_{11}& 0&0&0&b_{13}
\end{array}\right)=\setlength{\arraycolsep}{0.8pt}\left(\begin{array}{ccccc}
x& 0 & 0 & 0 & 0\\
\frac{\xi^{\frac{1}{3}}}{(rx+3q)^{\frac{1}{3}}} & \frac{x\xi^{\frac{1}{3}}}{(rx+3q)^{\frac{1}{3}}} & 0 & 0&0 \\
0 & \frac{2\xi^{\frac{2}{3}}}{(rx+3q)^{\frac{2}{3}}}&\frac{x\xi^{\frac{2}{3}}}{(rx+3q)^{\frac{2}{3}}} & 0 &0\\
0& 0 &\frac{3\xi}{rx+3q} & \frac{x\xi }{rx+3q}&0\\
0&0&0&0&\frac{(rx+3q)^{\frac{1}{3}}}{\xi^{\frac{1}{3}}}
\end{array}\right).
$$
Solving the system \eqref{pdesys} given in Proposition \ref{prop1} gives rise to a point transformation

$$
\bar{x}=x, \quad \bar{u}=xu .
$$
\end{example}

\begin{example}
Consider the class of non-linear fourth-order ODE

\begin{equation}\label{example2}
u^{(4)}=\frac{u_{xxx}^2}{u_{xx}}+\frac{4u_{xx}u_{xxx}}{u_x}-\frac{6u_{xx}^3}{u_x^2} .
\end{equation}

It can be verified that the function $f(x, u, p, q, r)=\frac{r^2}{q}+\frac{4qr}{p}-\frac{6q^3}{p^2}$ satisfies the conditions of Theorem \ref{thm2}. Moreover, since  $T^3_{13}= -1$ so $K=1 $ and $f$ is equivalent to the canonical form

\begin{equation}\label{example22}
\bar{u}^{(4)}=\frac{\bar{u}_{\bar{x}\bar{x}\bar{x}}^2}{\bar{u}_{\bar{x}\bar{x}}},
\end{equation} 
It can be checked that they have five symmetries, under point transformation and identical constant structure equations  given by
\begin{equation}
\begin{aligned}
 d \theta^1&=2\sqrt{15}\,\theta^1 \wedge \theta^2+9\,\theta^1 \wedge \theta^3+\frac{6}{5}\sqrt{15}\,\theta^1\wedge \theta^4-\frac{2}{3}\sqrt{15}\,\theta^1\wedge \theta^5-\theta^2\wedge \theta^5, \\
 d \theta^2&=\theta^1\wedge(-7\, \theta^2-\frac{6}{5}\sqrt{15}\, \theta^3-3\, \theta^4+\theta^5)+\theta^2\wedge ( 6\,\theta^3+\frac{3}{5}\sqrt{15}\,\theta^4-\frac{1}{3}\sqrt{15}\, \theta^5)-\theta^3\wedge \theta^5, \\
 d \theta^3&=\frac{4}{5}\sqrt{15}\,\theta^1 \wedge \theta^2-\theta^1 \wedge \theta^3-2\sqrt{15}\,\theta^2\wedge \theta^3-3\,\theta^2\wedge \theta^4+\frac{4}{3}\,\theta^2\wedge \theta^5-\theta^4\wedge \theta^5, \\
 d \theta^4&=\theta^1 \wedge(-\frac{1}{3}\, \theta^2+\frac{8}{15}\sqrt{15}\, \theta^3+ \theta^4+\frac{1}{9}\, \theta^5)
 +\theta^2 \wedge (4\,\theta^3+\frac{2}{5}\sqrt{15}\, \theta^4)+(\theta^3+\frac{1}{3}\sqrt{15}\,\theta^4)\wedge \theta^5,\\
 d \theta^5&=\theta^1 \wedge(\frac{21}{5}\, \theta^2+\frac{6}{5}\,\sqrt{15} \theta^3+\frac{9}{5}\, \theta^4-3\, \theta^5)-(\frac{4}{5}\sqrt{15}\,\theta^2-3\,\theta^3-\frac{3}{5}\sqrt{15}\,\theta^4)\wedge \theta^5,
\end{aligned}
\end{equation}

by choosing $\bar{J_8}=-\frac{\sqrt{15}\xi}{5\eta},~J_8=-\frac{\sqrt{15}(pr-3q^2)}{5qp}.$

The construction of the point transformation \eqref{point} requires the computation of the matrix entries \eqref{eqq} by using \eqref{example2} and \eqref{example22} as below 
$$
\left(\begin{array}{ccccc}
b_1 & 0 & 0 & 0& 0\\
b_2 & b_3 & 0 & 0& 0 \\
b_4 & b_5 & b_6 & 0&0 \\
b_7 & b_8 & b_9& b_{10}&0\\
b_{11}& 0&0&0&b_{13}
\end{array}\right)=\setlength{\arraycolsep}{0.5pt}\left(\begin{array}{ccccc}
\frac{\eta^3(pr-3q^2)^2}{\xi^2p^2q^3}& 0 & 0 & 0 & 0\\
-\frac{\eta^2(pr-3q^2)}{\xi qp^2}&\frac{\eta^2 (pr-3q^2)}{\xi p q^2} & 0 & 0&0 \\
-\frac{\eta(pr-3q^2)}{qp^2} & -\frac{3\eta}{p}& \frac{\eta}{q} & 0 &0\\
-\frac{\xi(pr-3q^2)}{qp^2}& -\frac{\xi(4pr-15q^2)}{p(pr-3q^2)} & -\frac{6\xi q}{pr-3q^2}& \frac{\xi p}{pr-3q^2}&0\\
\frac{\eta(pr-3q^2)}{q \xi p^2}&0&0&0&\frac{\eta(pr-3q^2)}{\xi p q}
\end{array}\right).
$$
Solving the system \eqref{pdesys} given in Proposition \ref{prop1} results in a point transformation

$$
\bar{x}=u, \quad \bar{u}=x .
$$
\end{example}

\begin{example}
Consider the class of non-linear fourth-order ODE

\begin{equation}\label{example3}
u^{(4)}=\frac{-24u_xu_{xx}u_{xxx}+18u_{xx}^3+4u_{xxx}^2u}{-6u_x^2+3u_{xx}u} .
\end{equation}

It can be checked that the function $f(x, u, p, q, r)=\frac{-24pqr+18q^3+4r^2u}{-6p^2+3qu}$ satisfies the conditions of Theorem \ref{thm3}, with identical constant structure equations, given  in equation \eqref{4/3c}, to the canonical form

\begin{equation}\label{example33}
\bar{u}^{(4)}=\frac{4}{3}\frac{u_{xxx}^2}{u_{xx}},
\end{equation} 

 with six symmetries under point transformation.
The construction of the point transformation \eqref{point2} requires the computation of the matrix entries\eqref{b2} by using \eqref{example3} and \eqref{example33} as below  
$$\setlength{\arraycolsep}{1pt}
\left(\begin{array}{cccccc}
b_1 & 0 & 0 & 0& 0&0\\
b_2 & b_3 & 0 & 0& 0&0 \\
b_4 & b_5 & b_6 & 0&0 &0\\
b_7 & b_8 & b_9& b_{10}&0&0\\
b_{11}& 0&0&0&b_{13}&0\\
b_{14}&b_{15}&b_{16}&0&b_{17}&b_{18}
\end{array}\right)=\setlength{\arraycolsep}{0.5pt}
\begin{pmatrix}
-\dfrac{\eta\,u\,a_{13}^2}{\Delta}
& 0 & 0 & 0 & 0 & 0
\\[8pt]
\chi_1&-\dfrac{\eta\,u\,a_{13}}{\Delta}& 0 & 0 & 0 & 0
\\[12pt]
-\dfrac{2\eta(3p^2-qu)}{u\Delta}&\dfrac{4\eta p}{\Delta}&-\dfrac{\eta u}{\Delta}& 0 & 0 & 0
\\[12pt]
\chi_2&\chi_3&\chi_4&-\dfrac{\eta u}{a_{13}\Delta}& 0 & 0
\\[14pt]
0 & 0 & 0 & 0 & a_{13} & 0
\\[10pt]
0 & 0 & 0 & 0 &\chi_5&
\dfrac{1}{a_{13}}
\end{pmatrix},$$
where \begin{equation*}\begin{aligned}\chi_1&=\frac{a_{13}}{3\Delta^2}(a_{13}\xi u\Delta+\eta( 6p(3p^2-2qu)+r u^2)),\\
\chi_2&=-\frac{2}{\Delta^2 u^2 a_{13}}(\xi a_{13}(up^2(6p^2-5uq)+u^3q^2)+\eta p(6p^2(-p^2+qu)+u^2(pr-3q^2))),\\
\chi_3&=\frac{2}{a_{13}u\Delta^2}(2a_{13}\xi pu\Delta-3\eta(p^2\Delta-u^2(\frac{2}{3}pr-q^2))),\\
\chi_4&=-\frac{a_{13}\xi u\Delta-\eta(6p^3-ru^2)}{a_{13}\Delta^2},\\
\chi_5&=\frac{2(a_{13}\xi u \Delta+\eta(6p(p^2-qu)+ru^2))}{3\eta u \Delta},\\
\Delta&=2p^2-qu.
\end{aligned}\end{equation*}

Solving the system \eqref{sys2} given in proposition \ref{prop2} yields in $a_{13}=-\frac{1}{x^2}$,\\$\eta=\frac{x^3}{u^3}(2p^2-qu),\quad \xi=\frac{x^4}{u^4}(-6p^2u+3qu^2-6pqux+6p^3x+rxu^2),\quad  g=\frac{xp}{u^2}+\frac{1}{u}$ and a point transformation can be expressed as

$$
\bar{x}=\frac{1}{x}, \quad \bar{u}=\frac{1}{xu} .
$$
\end{example}
\section{Conclusion}
Non-linearizable fourth-order ODEs under point transformation admitting a five-dimensional Lie point symmetry subalgebra play a central role in both theory and applications, for example, in mechanics (beam theory and elasticity), extensions of the Euler–Bernoulli theory. Symmetry reductions can lead to non-linearizable fourth-order canonical forms admitting a five-dimensional symmetry subalgebra, with a practical integration strategy. 

The equivalence problem for non-linearizable fourth-order ODEs with five-dimensional Lie symmetry subalgebra, under point transformation, was investigated by means of the Inductive Cartan Equivalence Method.

A primary difficulty in this work arises from the implementation of Cartan’s Equivalence Method, where the phenomenon of expression swell is encountered. As a result, the associated coframe and its corresponding computational procedures become highly complex, rendering the analysis of fourth-order ODEs significantly more challenging than that of lower-order ODEs.

We were able to overcome this difficulty by applying the Inductive Cartan Equivalence Method. This method begins by solving the equivalence problem under fiber-preserving transformations, thereby obtaining an initial adapted coframe with the power of making the expressions much simpler than before. Despite this simplification, the computations still remain complicated. Consequently, we also introduced a novel framework for Cartan's method by branching via differential relative invariants and introducing a chain of auxiliary functions.

Through this insightful framework we presented an invariant characterization of non-linearizable fourth-order ODEs under point transformation admitting a five-dimensional point symmetry Lie subalgebra. The results are stated as the main theorems in Section 3. We also proposed a method for constructing the associated point transformations based on invariant coframes, with an illustrative example for each branch.

\setlength{\arraycolsep}{0.5pt}
\section*{Appendix I}
\begin{table}[H]
\begin{adjustbox}{width=1.2\columnwidth,center}
\begin{tabular}{|c|c|c|}
\hline Algebra Type & Generators & The corresponding fourth-order equations \\
\hline $(24,5), \alpha=0$ & $\partial_x, \partial_u, x \partial_x+\alpha u \partial_u, x \partial_u, x^2 \partial_u$ & $u^{(4)}=K {u^{\prime \prime \prime}}^ {\frac{4}{3}}$ \\
\hline $(24,5), \alpha=b+1$ & $\partial_x, \partial_u, x \partial_x+\alpha u \partial_u, x \partial_u, x^2 \partial_u$ & $u^{(4)}=K u^{\prime \prime \prime \prime \frac{b-3}{b-2}}$ \\
\hline $(24,5), \alpha=2$ & $\partial_x, \partial_u, x \partial_x+\alpha u \partial_u, x \partial_u, x^2 \partial_u$ & $u^{(4)}=K u^{\prime \prime \prime 2}$ \\
\hline $(24,5), \alpha=1$ & $\partial_x, \partial_u, x \partial_x+\alpha u \partial_u, x \partial_u, x^2 \partial_u$ & $u^{(4)}=K u^{\prime \prime \prime \frac{3}{2}}$ \\
\hline $(24,5), \alpha=\frac{b+1}{b}$ & $\partial_x, \partial_u, x \partial_x+\alpha u \partial_u, x \partial_u, x^2 \partial_u$ & $u^{(4)}=K u^{\prime \prime \prime \frac{1-3 b}{1-2 b}}$ \\
\hline $(5,5)$ & $\partial_x, \partial_u, x \partial_x-u \partial_u, u \partial_x, x \partial_u$ & $u^{(4)}=\frac{5}{3} \frac{u^{\prime \prime \prime2}}{u^{\prime \prime}}+K u^{\prime \prime \frac{5}{3}}$ \\
\hline $(15,5)$ & $\partial_x, \partial_u, x \partial_x, u \partial_u, x^2 \partial_x$ & $
u^{(4)}=\frac{6 u^{\prime \prime} u^{\prime \prime \prime}}{u^{\prime}}-\frac{6 u^{\prime \prime 3}}{u^{\prime 2}}+K \frac{u^{\prime \prime 3}\left(3-2 \frac{u^{\prime} u^{\prime \prime \prime}}{u^{\prime \prime 2}}\right)^{\frac{3}{2}}}{u^{\prime 2}}
$\\
\hline $(25,5), \mathrm{r}=3$ & $\partial_x, \partial_u, x \partial_u, x^2 \partial_u, x \partial_x+\left(3 u+x^3\right) \partial_u$ & $
u^{(4)}=K e^{\frac{-u^{\prime \prime \prime}}{6}}
$ \\
\hline $(26,5), r=1$ & $\partial_x, \partial_u, x \partial_u, x \partial_x, u \partial_u$ & $
u^{(4)}=K \frac{u^{\prime \prime \prime 2}}{u^{\prime \prime}}
$ \\
\hline $(27,5), r=1$ & $\partial_x, \partial_u, x \partial_u, 2 x \partial_x+u \partial_u, x^2 \partial_x+x u \partial_u$ & $
u^{(4)}=\frac{4}{3} \frac{u^{\prime \prime \prime 2}}{u^{\prime \prime}}+K u^{\prime \prime \frac{7}{3}}
$\\
\hline $(21,5), r=3$ & $\partial_u, u \partial_u, x \partial_u, \xi_2(x) \partial_u, \xi_3(x) \partial_u$ & 
$u^{(4)}=\sum_{i=1}^2 A_i(x) u^{(i+1)}, \xi_k$ satisfy $\xi_k^{(4)}=\sum_{i=1}^2 A_i(x) \xi_k^{(i+1)}, k=2,3$\\
\hline $(6,6)$ & $\partial_x, \partial_u, x \partial_x, u \partial_x, x \partial_u, u \partial_u$ & $
u^{(4)}=\frac{5}{3} \frac{u^{\prime \prime \prime 2}}{u^{\prime \prime}}
$ \\
\hline$(28,6), r=1$ & $\partial_x, \partial_u, x \partial_x, x \partial_u, u \partial_u, x^2 \partial_x+x u \partial_u$ & $u^{(4)}=\frac{4}{3} \frac{u^{\prime\prime \prime2}}{y^{\prime \prime}}$ \\
\hline$(23,6), r=4 $& $\partial_x, u \partial_u, \eta_1(x) \partial_u, \eta_2(x) \partial_y, \eta_3(x) \partial_u, \eta_4(x) \partial_u$, &$ u^{(4)}=a_3 u^{\prime \prime \prime}+a_2 u^{\prime \prime}+a_1 u^{\prime}+a_0 u, ~a_i$ is constants for $i=0,1,2,3$ \\
\hline$(28,8), r=3$ & $\partial_x, \partial_u, x \partial_x, x \partial_u$ & $u^{(4)}=0$ \\
& $u \partial_u, x^2 \partial_x+3 x u \partial_u, x^2 \partial_u, x^3 \partial_u$ \\
\hline 
\end{tabular}
\end{adjustbox}
\caption{Classification of scalar fourth-order ODEs admitting real Lie algebras $(m,n)$, where $n= 5, 6, 8$ }
\end{table}
\begin{itemize}
\item By referring to algebra $(m, n)$, it is meant that $n$ is its dimension and $m$ the class of the algebra as given in \cite{OLVER}. 
\item $\xi_2,\xi_3$ form independent solution of the linear homogenous equation $u^{(n)}=\sum_{i=1}^{2} A_i(x) u^{(i+1)}$, that is, they satisfy  $\xi_k^{(4)}=\sum_{i=1}^{2} A_i(x) \xi_k^{(i+1)}, \quad k=2,3$, and $\eta_i(x),$ $i=1,2,3,4, $ are linearly independent solutions of $u^{(4)}=a_3 u^{\prime \prime \prime}+a_2 u^{\prime \prime}+a_1 u^{\prime}+a_0 u$, we refer the reader to \cite{OLVER}.
\end{itemize}

\section*{Appendix II}
The equivalence of \eqref{2.4} under the fiber-preserving transformation \begin{equation}\bar{x}=\phi(x),~\bar{u}=\psi(x,u)\end{equation} can be expressed as  

\begin{equation}\label{app}
\Phi^*\left(\begin{array}{c}
d\bar{u}-\bar{p}d\bar{x} \\
d\bar{p}-\bar{q}d\bar{x} \\
d\bar{q}-\bar{r}d\bar{x}\\
d\bar{r}-\bar{f}(\bar{x},\bar{u},\bar{p},\bar{q},\bar{r})d\bar{x}\\
d\bar{x}
\end{array}\right)=\left(\begin{array}{ccccc}
a_1 & 0 & 0 &0 &0\\
a_2 & a_3 & 0 &0&0 \\
a_4 & a_5 & a_6& 0&0\\
a_7 & a_8&a_9&a_{10}&0\\
0&&0&0&a_{13}
\end{array}\right)\left(\begin{array}{c}
du-pdx\\
dp-qdx\\
dq-rdx\\
dr-f(x,u,p,q,r)dx\\
dx
\end{array}\right).
\end{equation}
The structure equations for the first loop is similar to \eqref{e2.9} with the same essential torison coefficients as in \eqref{2.10} that also can be translated to $-1$ to get \eqref{2.11}.\\ 

In the \textit{second iteration} of the reduction scheme, absorption leads to structure equations with nonzero essential torisons $T^3_{35},~T^4_{45}$ given explicitly by \begin{equation}T^3_{35}=\frac{a_9a_{13}^2-3a_5a_{13}+3a_2}{a_1}, ~T^4_{45}=\frac{-a_9a_{13}^2-3a_5a_{13}^2+5a_2a_{13}+I_0a_1}{a_1a_{13}},\end{equation} which can be translated to $0$ by normalizing \begin{equation}a_2=a_9a_{13}^2-\frac{1}{2}\frac{a_1I_0}{a_{13}}, a_5=\frac{4}{3}a_9a_{13}-\frac{1}{2}\frac{a_1I_0}{a_{13}^2},\end{equation} where $I_0=-f_r$.\\

In the \textit{third iteration} of the reduction scheme, absorption leads to structure equations with nonzero essential torisons  $T^3_{25},~T^4_{35}$ given by \begin{equation}\begin{aligned}
 T^3_{25}&=-\frac{19}{18}\frac{a_9a_{13}I_0}{a_1}+\frac{I_1}{a_{13}^2}-\frac{7}{3}\frac{a_4}{a_1}+\frac{a_8a_{13}}{a_1}+\frac{8}{9}\frac{a_9^2a_{13}^4}{a_1^2},\\ T^4_{35}&=-\frac{7}{6}\frac{a_9a_{13}I_0}{a_1}+\frac{I_2}{a_{13}^2}-\frac{a_4}{a_1}-\frac{a_8a_{13}}{a_1}+\frac{4}{3}\frac{a_{9}^2a_{13}^4}{a_1^2},\end{aligned}\end{equation} that can be translated to $0$ by normalizing \begin{equation}\begin{aligned}
 a_4&=\frac{2}{3}\frac{a_9^2a_{13}^4}{a_1}-\frac{2}{3}a_9a_{13}I_0+\frac{3}{10}\frac{a_1}{a_{13}^2}I_1+\frac{3}{10}\frac{a_1}{a_{13}^2}I_2,\\
a_8&=\frac{2}{3}\frac{a_9^2a_{13}^3}{a_1}-\frac{1}{2}a_9I_0-\frac{3}{10}\frac{a_1}{a_{13}^3}I_1+\frac{7}{10}\frac{a_1}{a_{13}^3}I_2,\end{aligned}\end{equation} where $I_1=\frac{1}{4}I_0^2-\frac{1}{6}\hat{D}_xI_0, ~I_2= -\frac{1}{2}\hat{D}_xI_0-f_q,$ 
 the  remaining parameters are $a_1,~a_7,~a_9,~a_{13}.$
Substituting the identity values of the group parameters $a_1=1~,a_7=0,~a_{9}=0,~a_{13}=1$ in the last coframe yields \eqref{2.1} and \eqref{2.2}.

\end{document}